\documentclass{amsart}
\usepackage{amsmath,amsfonts,amsthm,amssymb,amsrefs}
\usepackage{tikz}
\usetikzlibrary{arrows}

\newtheorem{theorem}{Theorem}

\newtheorem{lemma}{Lemma}
\newtheorem{prop}{Proposition}

\newtheorem{claim}{Claim}
\newtheorem{question}{Question}

\newcommand{\C}{\mathcal C}
\newcommand{\EC}{\rm EC}
\newcommand{\SO}{\rm SO}
\newcommand{\ESC}{\rm ESC}
\newcommand{\A}{\mathcal A}
\newcommand{\FL}{\mathcal F \mathcal F}
\newcommand{\FF}{\mathcal F\mathcal F}

\begin{document}
\title{Cylinder curves in finite holonomy flat metrics}
\date{}
\author[S-W.~Fu]{Ser-Wei Fu }
\author[C.J.~Leininger]{Christopher Leininger}

\address{S-W.~Fu, National Center for Theoretical Sciences, No. 1 Sec. 4 Roosevelt Rd., National Taiwan University, Taipei, 106, Taiwan, email: swfu@ncts.ntu.edu.tw}
\address{C.J.~Leininger, Department of Mathematics, Rice University, Houston, TX, email: cjl12@rice.edu}

\begin{abstract}
For an orientable surface of finite type equipped with a flat metric with holonomy of finite order $q$, the set of maximal embedded cylinders can be empty, non-empty, finite, or infinite. The case when $q\leq 2$ is well-studied as such surfaces are (semi-)translation surfaces. Not only is the set always infinite, the core curves form an infinite diameter subset of the curve complex. In this paper we focus on the case $q\geq 3$ and construct examples illustrating a range of behaviors for the embedded cylinder curves. We prove that if $q \geq 3$ and the surface is {\em fully punctured}, then the embedded cylinder curves form a finite diameter subset of the curve complex.  The same analysis shows that the embedded cylinder curves can only have infinite diameter when the metric has a very specific form.  Using this we characterize precisely when the embedded cylinder curves accumulate on a point in the Gromov boundary.
\end{abstract}

\maketitle


\section{Introduction}\label{1:Intr}

We study the embedded cylinders in an orientable  surface $S$ of finite type, equipped with a {\em flat metric}: that is, a non-positively curved Euclidean cone metric, whose metric completion is also a Euclidean cone metric of finite area. By a \emph{cylinder} for a flat metric $\varphi$ on $S$, we mean an isometrically immersed open  Euclidean cylinder.  If the cylinder is embedded, then the (isotopy class of its) core curve is a simple closed curve in $S$ called an {\em embedded cylinder curve}.  The set of all embedded cylinder curves is denoted $\EC(\varphi)$, which we view as a subset of $\C(S)$, the {\em curve complex of $S$}; see Section \ref{2:Back}.
  
To describe the situation of interest for us, we recall that the {\em holonomy} of a flat metric $\varphi$ on $S$ is the image of the holonomy homomorphism $\pi_1(S^\circ) \to \SO(2)$ defined by parallel transport around loops on $S^\circ$, the surface obtained by removing all cone points of $\varphi$.   When the holonomy has order $q \leq 2$, the metric comes from a semi-translation structure, so every cylinder is embedded and a consequence of Masur's work \cite{Mas} is that $\EC(\varphi) \subset \C(S)$ has infinite diameter.
In this paper, we are interested in flat metrics with holonomy of finite order $q \geq 3$.

We construct a number of examples of flat metrics $\varphi$ for which the holonomy has order $q$, with $2 < q < \infty$, and for which $\EC(\varphi)$ exhibits a variety of behaviors.  Specifically, we construct examples of such metrics $\varphi$ on closed surfaces for which $\EC(\varphi)$ is empty, non-empty, finite, and infinite; see Section~\ref{3:Exam}.  Moreover, among the examples where $\EC(\varphi)$ is infinite, we find some for which $\EC(\varphi)$ has infinite diameter in $\C(S)$.
Despite this last family of examples, in Section~\ref{4:Main}, we show that when $\varphi$ is {\em fully punctured}, that is it has no cone points, then the diameter in the curve complex is {\em always} finite.

\begin{theorem} \label{T:main}
Let $\varphi$ be a fully punctured flat metric on a surface $S$ with finite holonomy of order at least $3$.  Then $\EC(\varphi) \subset \C(S)$ has finite diameter.
\end{theorem}

The ideas in the proof of this theorem can be used to show that all examples for which $\EC(\varphi)$ has infinite diameter must arise from the kind of construction we give in Section~\ref{S:slit construction} and mentioned in the previous paragraph; see Theorem~\ref{T:the only way}.  In fact, the examples in Theorem~\ref{T:6 flat construction} in Section~\ref{S:slit construction} and the Theorem~\ref{T:the only way} together characterize when the set can accumulate on the Gromov boundary.  The proof of Theorem~\ref{T:main} can also be modified to prove a similar result for the set of embedded saddle connections; see Theorem~\ref{T:ESC}.

Flat metrics with finite order holonomy arise naturally as a generalization of semi-translation surfaces and share a number of properties. For such metrics, the set of cylinder curves is an important tool in identifying the metric up to affine deformation, leading to rigidity results of the second author with Duchin and Rafi \cite{DLR} and Loving \cite{Loving}.  For semi-translation surfaces, combinatorial and geometric properties of cylinder curves and saddle connections, viewed as subsets of the curve graph and arc graph have recently been recently studied by Minsky-Taylor \cite{MinTay}, Nguyen \cite{Nguyen}, Tang-Webb \cite{TW}, Pan \cite{Pan} and Disarlo-Randecker-Tang \cite{DRT}, as well as forthcoming work of Tang \cite{TangPrep}.  The connection between flat metrics and the curve graph has its origin in the work of Masur and Minsky \cite{MM1} (see also Bowditch \cite{Bow}), while more generally, cylinders in flat metrics arose naturally in complex analysis via extremal problems (see Strebel \cite{Str}) and in dynamics of rational billiards via periodic billiard trajectories (see, for example, Masur \cite{Mas} and Boshernitzan-Galperin-Kr\"{u}ger-Troubetzkoy \cite{BGKT}).\\


\noindent \textbf{Acknowledgements.} The authors thank the Fields Institute for their hospitality during the thematic program on {\em Teichm\"{u}ller Theory and its Connections to Geometry, Topology and Dynamics} where the authors began contemplating this problem. We also thank the National Center of Theoretical Sciences and National Taiwan University for their support during the second author's visit to Taiwan.  They also thank Marissa Loving, Robert Tang, and Valentina Disarlo for interesting and helpful conversations related to this work.  Leininger was partially supported by NSF grants DMS 1811518 and DMS 1107452, 1107263, 1107367 ``RNMS: GEometric structures And Representation varieties" (the GEAR Network).  Finally, we would like to thank the anonymous referee for their careful reading of the first version of this paper, for pointing out the reference \cite{Vorobets} which led to the more general construction in Theorem~\ref{T:6 flat construction}, and for suggesting we include a precise characterization of when embedded cylinder curves limit to a point in the Gromov boundary.


\section{Notation}\label{2:Back}

A Euclidean cone metric $\varphi$ on a surface $S$ is a metric which is locally isometric to the Euclidean plane away from a finite number of points.  At these points, the metric has a cone singularity with well-defined cone angle $> 0$.  The metric is non-positively curved (locally CAT(0)) if all cone points have cone angle greater than $2\pi$.  

The metric completion of a non-positively curved Euclidean cone metric may not even be a surface in general (e.g.~the metric completion of the universal cover of the Euclidean plane minus a point), but for the purposes of this paper, we will only consider those metrics where it is a surface, and is again a Euclidean cone metric of finite area.  We call such a metric a {\em flat metric} on $S$.

Given a flat metric $\varphi$ on $S$, the associated fully punctured flat metric is the restriction $\varphi^\circ$ of $\varphi$ to the surface $S^\circ$ obtained by removing all cone points of $\varphi$ from $S$.  The holonomy homomorphism $\pi_1(S^\circ) \to \SO(2)$ is obtained by parallel translating around loops based at some point of $S^\circ$, and the holonomy is defined to be the image.  When the order of the holonomy $q$ is $1$ or $2$, the flat metric comes from a translation or semi-translation surface.  The primary focus of this paper is the case where the holonomy has finite order $q$ {\em at least $3$}.  If the holonomy has order dividing $q$, we say that the metric is a {\em $q$--flat metric}.

Given $q \geq 1$, a meromorphic $q$--differential on a closed Riemann surface with poles of order at most $q-1$ determines a Euclidean cone metric with order of holonomy dividing $q$.  Puncturing at the poles (if any) determines a $q$--flat metric on the resulting surface.  Specifically, integrating a $q^{th}$ root of the differential away from the singularities determines a {\em preferred coordinate}, well defined up to translation and rotation through angles which are integral multiples of $\frac{2\pi}q$ (depending on the choice of $q^{th}$ root).  These coordinates therefore define a Euclidean metric on the complement of the zeros/poles with holonomy having order dividing $q$.  This metric extends over a zero of order $k \geq 1-q$ (a pole if $k < 0$) to produce cone singularities with cone angles $2\pi + \frac{2\pi k }q$ (c.f.~Bankovic \cite{Ban}).  We will not use this complex analytic perspective, but mention it for context.

The special case of $q =2$ arises naturally in Teichm\"uller theory.  Specifically, given a $2$--flat metric, we may choose an atlas of Euclidean coordinate charts for which the transition functions are translations, possibly composed with rotation through angle $\pi$.  The structure determined by such an atlas is sometimes called a {\em semi-translation structure}, and is equivalent to the data of a complex structure and integrable holomorphic quadratic differential (for a $1$--flat metric, we can take the transition functions to be translations only, and we get a {\em translation structure} or a holomorphic $1$--form).  The importance of this case in Teichm\"uller theory is that deforming the coordinates by the maps $g_t(x+iy) = e^tx + ie^{-t}y$, for any $t \in \mathbb R$, defines a new $2$--flat metric, and the 1-parameter family of such deformations is precisely a Teichm\"uller geodesic.  We also note that the holonomy condition allows for a parallel line field in any given direction, which in turn integrates to a measured singular foliation; see e.g.~\cite{Gardiner}.

We can also concretely describe $q$--flat metrics by gluing sides of finite number of compact Euclidean polygons by isometries that are compositions of translations and rotations through angles that are integral multiples of $\frac{2\pi}{q}$.  For example, consider the equilateral triangle on the right in Figure~\ref{NoCyl12gon} with each side subdivided into four equal length arcs, which we view as a 12-gon where some of the interior angles are equal to $\pi$.  The pairing of these sides of this 12-gon described by the labels in the figure can be realized by gluings given by translations and rotations through angles $\frac{k\pi}{3} = \frac{2 k \pi}{6}$, for $k \in \mathbb Z$.  The result is a genus $3$ surface with a flat metric having holonomy of order $6$, and a single cone point of cone angle $10\pi$. 

Given a fully punctured flat metric $\varphi$ on $S$ with finite holonomy of order $q$ the {\em holonomy almost trivializing cover} $p_0 \colon (S_0,\varphi_0) \to (S,\varphi)$ is the locally isometric cover corresponding to the preimage of the subgroup $\{\pm I\}< \SO(2)$ (with the pull back metric $\varphi_0 = p^*\varphi$); see Loving \cite{Loving}.  The cover $p$ is a cyclic cover of degree $q_0 = q$ if $q$ is odd and $q_0 = \frac{q}2$ if $q$ is even.  
If $\varphi$ is not fully punctured, this cover still makes sense as a branched covering, branched over (some of) the cone points.   If we build the flat metric by gluing polygons, we can construct the holonomy almost trivializing cover from $q_0$ rotated copies of the polygons. See Figure~\ref{NoCyl12gon} on the left for the example of the cover associated to the genus $3$ surface described above (and illustrated on the right).

\vspace{.3cm}

\noindent {\bf Remark.} The holonomy almost trivializing cover is non-trivial if and only if $q\geq 3$, which is an obvious, yet key point in our proof of Theorem~\ref{T:main}. 

\vspace{.3cm}

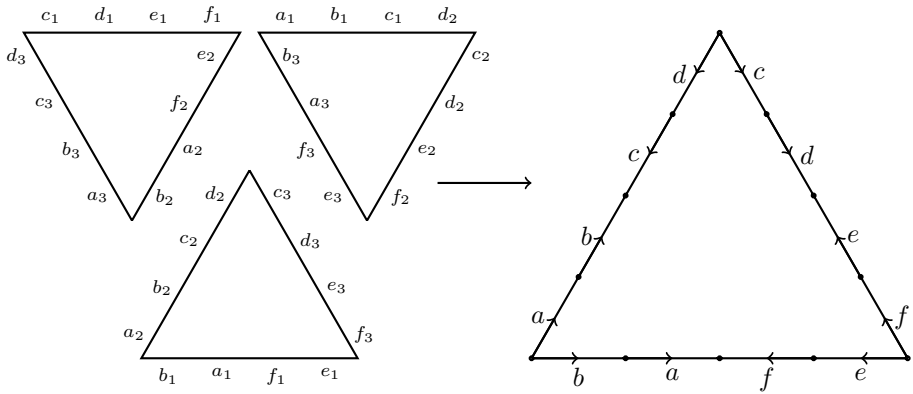
\begin{figure}[ht]
\begin{tikzpicture}[scale=1.25]

\draw[thick] (-4.5,1.134) -- node[left]{\scriptsize $d_2$} (-4.7875,0.634) -- node[left]{\scriptsize $c_2$} (-5.075,0.134) -- node[left]{\scriptsize $b_2$} (-5.3625,-0.366) -- node[left]{\scriptsize $a_2$} (-5.65,-0.866) -- node[below]{\scriptsize $b_1$} (-5.075,-0.866) -- node[below]{\scriptsize $a_1$} (-4.5,-0.866) -- node[below]{\scriptsize $f_1$} (-3.925,-0.866) -- node[below]{\scriptsize $e_1$} (-3.35,-0.866) -- node[right]{\scriptsize $f_3$} (-3.6375,-0.366) -- node[right]{\scriptsize $e_3$} (-3.925,0.134) -- node[right]{\scriptsize $d_3$} (-4.2125,0.634) -- node[right]{\scriptsize $c_3$} (-4.5,1.134);
\draw[thick] (-5.75,0.6) -- node[left]{\scriptsize $a_3$} (-6.0375,1.1) -- node[left]{\scriptsize $b_3$} (-6.325,1.6) -- node[left]{\scriptsize $c_3$} (-6.6125,2.1) -- node[left]{\scriptsize $d_3$} (-6.9,2.6) -- node[above]{\scriptsize $c_1$} (-6.325,2.6) -- node[above]{\scriptsize $d_1$} (-5.75,2.6) -- node[above]{\scriptsize $e_1$} (-5.175,2.6) -- node[above]{\scriptsize $f_1$} (-4.6,2.6) -- node[left]{\scriptsize $e_2$} (-4.8875,2.1) -- node[left]{\scriptsize $f_2$} (-5.175,1.6) -- node[right]{\scriptsize $a_2$} (-5.4625,1.1) -- node[right]{\scriptsize $b_2$} (-5.75,0.6);
\draw[thick] (-3.25,0.6) -- node[left]{\scriptsize $e_3$} (-3.5375,1.1) -- node[left]{\scriptsize $f_3$} (-3.825,1.6) -- node[right]{\scriptsize $a_3$} (-4.1125,2.1) -- node[right]{\scriptsize $b_3$} (-4.4,2.6) -- node[above]{\scriptsize $a_1$} (-3.825,2.6) -- node[above]{\scriptsize $b_1$} (-3.25,2.6) -- node[above]{\scriptsize $c_1$} (-2.675,2.6) -- node[above]{\scriptsize $d_2$} (-2.1,2.6) -- node[right]{\scriptsize $c_2$} (-2.3875,2.1) -- node[right]{\scriptsize $d_2$} (-2.675,1.6) -- node[right]{\scriptsize $e_2$} (-2.9625,1.1)-- node[right]{\scriptsize $f_2$} (-3.25,0.6);
\draw[thick,->] (-2.5,1) -- (-1.5,1);

\draw[thick] (-1.5,-0.866) circle (0.02) -- (-1,0) circle (0.02) -- (-0.5,0.866) circle (0.02) -- (0,1.732) circle (0.02) -- (0.5,2.598) circle (0.02) -- (1,1.732) circle (0.02) -- (1.5,0.866) circle (0.02) -- (2,0) circle (0.02) -- (2.5,-0.866) circle (0.02) -- (1.5,-0.866) circle (0.02) -- (0.5,-0.866) circle (0.02) -- (-0.5,-0.866) circle (0.02) -- (-1.5,-0.866);

\draw[thick,->] (-1.5,-0.866) -- (-1.25,-0.433) node[left]{$a$};
\draw[thick,->] (-1,0) -- (-0.75,0.433) node[left]{$b$};
\draw[thick,->] (-0.5,-0.866) -- (0,-0.866) node[below]{$a$};
\draw[thick,->] (-1.5,-0.866) -- (-1,-0.866) node[below]{$b$};

\draw[thick,->] (0.5,2.598) -- (0.75,2.162) node[right]{$c$};
\draw[thick,->] (1,1.732) -- (1.25,1.299) node[right]{$d$};
\draw[thick,->] (0,1.732) -- (-0.25,1.299) node[left]{$c$};
\draw[thick,->] (0.5,2.598) -- (0.25,2.162) node[left]{$d$};

\draw[thick,->] (2.5,-0.866) -- (2,-0.866) node[below]{$e$};
\draw[thick,->] (1.5,-0.866) -- (1,-0.866) node[below]{$f$};
\draw[thick,->] (2,0) -- (1.75,0.433) node[right]{$e$};
\draw[thick,->] (2.5,-0.866) -- (2.25,-0.433) node[right]{$f$};

\end{tikzpicture}
\caption{A flat metric with holonomy of order $6$ and the (degree $3$) holonomy almost trivializing cover.}
\label{NoCyl12gon}
\end{figure}

If a closed curve $\gamma$ on a $q$-flat metric has a geodesic representative that does not contain any cone points, then $\gamma$ is called a {\em cylinder curve} and lies inside a unique maximal immersed open Euclidean cylinder, or just a {\em maximal cylinder}. Combining Masur's result \cite{Mas} with the holonomy almost trivializing cover, every $q$-flat metric has infinitely many homotopy classes of cylinder curves. In this paper we are focused on the homotopy classes of embedded cylinder curves, or the set of embedded maximal cylinders of a $q$-flat metric $\varphi$ denoted by $\EC(\varphi)$. 

Any geodesic in a $q$-flat metric that connects two punctures/cone points (not necessarily distinct) with interior disjoint from the cone points is called a {\em saddle connection}. We use $\ESC(\varphi)$ to denote the set of saddle connections with embedded interior in $\varphi$.

The curve graph $\mathcal{C}(S)$ and the arc graph $\A(S)$ are simplicial graphs whose vertices correspond to isotopy classes of essential simple closed curves and properly embedded arcs, respectively.  In each case, vertices are joined by an edge if the isotopy classes admit disjoint representatives.
The ``endpoints'' of arcs are punctures, hence any saddle connection of a fully punctured surface is an arc. 
We view $\EC(\varphi) \subset \C(S)$ and $\ESC(\varphi) \subset \A(S)$ (the latter when $\varphi$ is fully punctured), as sets of vertices in the respective graphs.  The diameter of such a subset is the supremum of edge-path distances between any two points in it.


\section{Examples}\label{3:Exam}

In this section we provide several examples of $q$-flat metrics whose cylinder sets exhibit a range of behaviors.

\subsection{Finitely many embedded cylinders.}

The following fact from classical Euclidean geometry proves quite useful in limiting embedded cylinders in a flat metric.

\begin{prop}\label{Classical}
Let $A$ and $B$ be distinct points in $\mathbb{R}^2$. The set 
\[
I_{AB}(\theta) = \{ C\in\mathbb{R}^2 \mid \angle ACB=\theta \},
\]
for some fixed $0 < \theta < \pi$, is the two circular arcs of central angle $2(\pi-\theta)$ connecting $A$ and $B$. 
\end{prop}

Proposition~\ref{Classical} will be used to determine whether a geodesic has a self-intersection with a prescribed intersection angle; see the next two constructions.


\subsubsection{Octagon} \label{S:octagon}

The first example is a $4$--flat metric $\varphi$ on a closed genus $2$ surface $S$.  Starting with a regular octagon $P$, we identify sides by isometry as indicated by the labeling in Figure~\ref{Octagon}.  We note that this is the picture of a genus $2$ surface commonly found in a first course in topology, as opposed to the ``usual octagon example'' often found in studying translation structures (where one identifies opposite sides).
As with the translation example, this $4$--flat metric has a single cone point with cone angle $6\pi$.

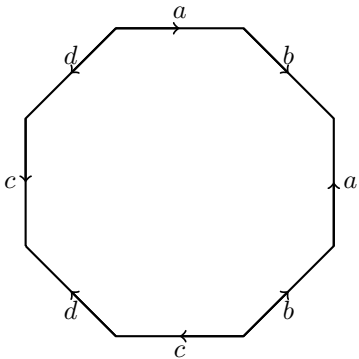
\begin{figure}[ht]
\begin{tikzpicture}[scale=1.2]

\draw[thick] (-2,-1) -- (-2,0.414) -- (-1,1.414) -- (0.414,1.414) -- (1.414,0.414) -- (1.414,-1) -- (0.414,-2) -- (-1,-2) -- (-2,-1);
\draw[thick,->] (-2,0.414) -- (-2,-0.293) node[left]{$c$};
\draw[thick,->] (0.414,-2) -- (-0.293,-2) node[below]{$c$};
\draw[thick,->] (1.414,-1) -- (1.414,-0.293) node[right]{$a$};
\draw[thick,->] (-1,1.414) -- (-0.293,1.414) node[above]{$a$};
\draw[thick,->] (0.414,1.414) -- (0.914,0.914) node[above]{$b$};
\draw[thick,->] (0.414,-2) -- (0.914,-1.5) node[below]{$b$};
\draw[thick,->] (-1,-2) -- (-1.5,-1.5) node[below]{$d$};
\draw[thick,->] (-1,1.414) -- (-1.5,0.914) node[above]{$d$};

\end{tikzpicture}
\caption{A 4-flat metric on a genus $2$ surface.}
\label{Octagon}
\end{figure}

The metric $\varphi$ has three embedded cylinders, which are shown in Figure~\ref{Octagon:Cyl}.  In fact, these are the only maximal embedded cylinders.
\begin{prop} \label{P:only 3} The flat metric $\varphi$ above has exactly three maximal cylinders.  That is, $\EC(\varphi)$ consists of exactly three curves.
\end{prop}
\begin{proof}  Suppose $\gamma$ is the core of some cylinder on $(S,\varphi)$, and is thus a closed geodesic on $(S,\varphi)$ containing no cone points.  Our goal is to show that $\gamma$ (and hence the cylinder) can only be embedded if it is one of the three cylinders already described.  To this end, observe that $\gamma$ cuts through $P$ in a collection of straight line segments, $\gamma_1,\ldots,\gamma_k$, which identify to $\gamma$ upon gluing the paired sides.  The proof reduces to an argument about the angles that these segments make with the sides, and in referring to these angles, we will always mean the angle in $(0,\frac{\pi}2]$.

First, we claim that if one of the segments $\gamma_i$ meets a side at a point $x$ making angle at least $\frac{\pi}4$, then $\gamma$ will have a self-intersection.  To see this, let $y$ be the point identified to $x$ along the identified side, and let $\gamma_{i+1}$ be the segment meeting the side at the point $y$.  By inspection, one of the semi-circles in $I_{xy}(\frac{\pi}2)$ is entirely contained in $P$, and the segments $\gamma_i,\gamma_{i+1}$ intersect at a point of $I_{xy}(\frac{\pi}2)$; see Figure~\ref{Octagon:Semicircle}.  It follows that $\gamma$ has a self-intersection, and so it is not embedded.

Next, we observe that any line segment that connects non-adjacent sides of $P$ makes angle at least $\frac{\pi}4$ at the point of intersection with at least one of the sides.  Therefore, by the previous paragraph, if $\gamma$ is embedded it must be that each $\gamma_i$ connects adjacent sides of $P$.  Each pair of adjacent sides determines a triangle in $P$ as in Figure~\ref{Octagon:Cyl} which is contained in exactly one of the three given embedded cylinders, and each segment $\gamma_i$ must therefore project into a unique cylinder and makes an angle strictly less than $\frac{\pi}8$ with the core geodesics of that cylinder.  If $\gamma$ is entirely contained in one of the cylinders, then it must be a core geodesic of that cylinder and we are done.  Otherwise, there must be points of $\gamma$ which are contained in two of these cylinders.  However, intersecting cylinders have core geodesics that meet at angle $\frac{\pi}4$, and so when the switch happens, the angle with the core geodesic must switch from being strictly less than $\frac{\pi}8$ to being strictly greater than $\frac{\pi}8$, a contradiction.
\end{proof}


\begin{figure}[ht]
\begin{tikzpicture}[scale=0.8]

\draw[thick] (-2,-1) -- (-2,0.414) -- (-1,1.414) -- (0.414,1.414) -- (1.414,0.414) -- (1.414,-1) -- (0.414,-2) -- (-1,-2) -- (-2,-1);
\draw[fill,red] (-1,1.414) -- (0.414,1.414) -- (1.414,0.414) -- (-1,1.414);
\draw[fill,red] (1.414,0.414) -- (1.414,-1) -- (0.414,-2) -- (1.414,0.414);

\draw[thick] (3,-1) -- (3,0.414) -- (4,1.414) -- (5.414,1.414) -- (6.414,0.414) -- (6.414,-1) -- (5.414,-2) -- (4,-2) -- (3,-1);
\draw[fill,magenta] (4,1.414) -- (3,0.414) -- (3,-1) -- (4,1.414);
\draw[fill,magenta] (3,-1) -- (4,-2) -- (5.414,-2) -- (3,-1);

\draw[thick] (8,-1) -- (8,0.414) -- (9,1.414) -- (10.414,1.414) -- (11.414,0.414) -- (11.414,-1) -- (10.414,-2) -- (9,-2) -- (8,-1);
\draw[fill,blue] (10.414,1.414) -- (8,0.414) -- (9,1.414) -- (10.414,1.414);
\draw[fill,blue] (8,0.414) -- (8,-1) -- (9,-2) -- (8,0.414);
\draw[fill,blue] (10.414,1.414) -- (11.414,0.414) -- (11.414,-1) -- (10.414,1.414);
\draw[fill,blue] (9,-2) -- (10.414,-2) -- (11.414,-1) -- (9,-2);

\end{tikzpicture}
\caption{Cylinders in the octagon example.}
\label{Octagon:Cyl}
\end{figure}
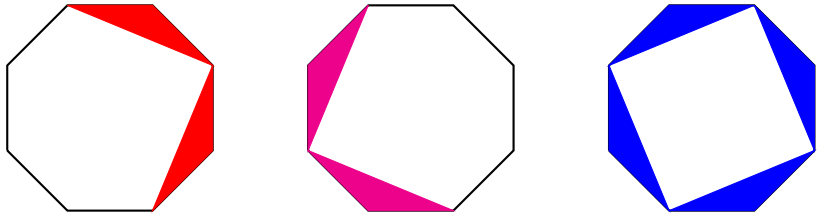


\begin{figure}[ht]
\begin{tikzpicture}[scale=1.0]

\draw[fill,yellow] (-1,-2) -- (0.414,-2) -- (-2,0.414) -- (-2,-1);
\draw[fill,yellow] (0.414,-2) arc (-45:135:1.707);
\draw[fill,white] (-1,-2) arc (-45:135:0.707);

\draw[thick] (-2,-1) -- (-2,0.414) -- (-1,1.414) -- (0.414,1.414) -- (1.414,0.414) -- (1.414,-1) -- (0.414,-2) -- (-1,-2) -- (-2,-1);
\draw[thick,->] (-2,0.414) -- (-2,-0.293) node[left]{$c$};
\draw[thick,->] (0.414,-2) -- (-0.293,-2) node[below]{$c$};
\draw[thick,->] (1.414,-1) -- (1.414,-0.293) node[right]{$a$};
\draw[thick,->] (-1,1.414) -- (-0.293,1.414) node[above]{$a$};
\draw[thick,->] (0.414,1.414) -- (0.914,0.914) node[above]{$b$};
\draw[thick,->] (0.414,-2) -- (0.914,-1.5) node[below]{$b$};
\draw[thick,->] (-1,-2) -- (-1.5,-1.5) node[below]{$d$};
\draw[thick,->] (-1,1.414) -- (-1.5,0.914) node[above]{$d$};

\draw[blue] (-1,-2) arc (-45:135:0.707);
\draw[red] (0.2,-2) arc (-45:135:1.54);
\draw[blue] (0.414,-2) arc (-45:135:1.707);

\draw[red,thick] (0.2,-2) -- (0.74,0.7);
\draw[red,thick] (-2,0.2) -- (0.97,-0.394);

\filldraw (.2,-2) circle[radius=.04];
\node at (0.2,-2.2) {$x$};
\node at (.95,.4) {$\gamma_i$};

\filldraw (-2,0.2) circle[radius=.04];
\node at (-2.2,0.2) {$y$};
\node at (-.3,.1) {$\gamma_j$};

\end{tikzpicture}
\caption{Self-intersections occur inside the shaded region when the angle with edge $c$ is at least $\frac{\pi}{4}$.}
\label{Octagon:Semicircle}
\end{figure}
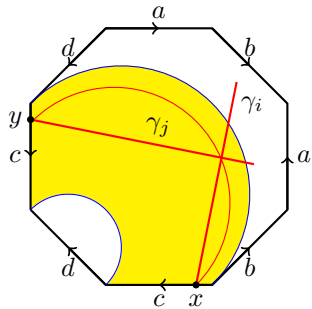



A similar argument can be used to prove that every regular $4g$-gon with the same gluing pattern per four consecutive edges is a $2g$--flat metric that has exactly $g+1$ embedded cylinder curves. 


\subsubsection{Cylinder-free region} \label{S:Cyl-Free}

Here we construct ``building blocks" that can be used to construct {\em cylinder-free regions}: subsets of a surface with a flat metric which are disjoint from all embedded cylinders.
For example, the part of the octagon that remains after removing the three cylinders is the cylinder-free region in the example from Section~\ref{S:octagon}.

Consider the non-positively curved Euclidean cone surface $\mathcal B$ (with boundary) illustrated in Figure~\ref{NoCylBlock2}. The surface $\mathcal B$ is constructed from an equilateral triangle, which will ultimately be the cylinder-free part, union with a \emph{buffer} region bounded by the circular arc of central angle $\frac{4\pi}{3}$ tangent to two sides of the triangle at their endpoints.  The two sides of the triangle on the boundary of the region are subdivided and identified as indicated by isometries (which are rotations about some point through cone angle $\pm \frac{\pi}3$).  We call $\mathcal B$ a {\em building block}.

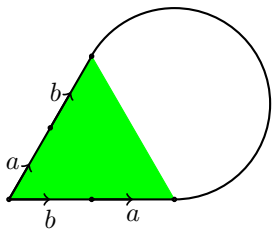
\begin{figure}[ht]
\begin{tikzpicture}[scale=1.1]

\draw[fill,green] (-1.5,-0.866) -- (-0.5,0.866) -- (0.5,-0.866);
\draw[thick] (-1.5,-0.866) circle (0.02) -- (-1,0) circle (0.02) -- (-0.5,0.866) circle (0.02);
\draw[thick] (0.5,-0.866) circle (0.02) -- (-0.5,-0.866) circle (0.02) -- (-1.5,-0.866);
\draw[thick,->] (-1.5,-0.866) -- (-1.25,-0.433) node[left]{$a$};
\draw[thick,->] (-1,0) -- (-0.75,0.433) node[left]{$b$};
\draw[thick,->] (-0.5,-0.866) -- (0,-0.866) node[below]{$a$};
\draw[thick,->] (-1.5,-0.866) -- (-1,-0.866) node[below]{$b$};
\draw[thick] (0.5,-0.866) arc (-90:150:1.157);

\end{tikzpicture}
\caption{The building block $\mathcal B$ with a triangular cylinder-free region.}
\label{NoCylBlock2}
\end{figure}

\begin{prop} Suppose the building block $\mathcal B$ is locally isometrically immersed in a flat surface $(S,\varphi)$.  Then the image of the triangular region is part of the cylinder-free region of $(S,\varphi)$.
\end{prop}
\begin{proof}  Suppose $\gamma$ is a closed non-singular geodesic in $(S,\varphi)$ that intersects the (image of) the triangular part of the building block, $\mathcal B$.  The goal is to show that $\gamma$ must have a self-intersection.  Consider a maximal arc $\gamma'$ of $\gamma$ that lifts to $\mathcal B$ meeting the triangular region: lifting, we view $\gamma'$ as a geodesic arc that enters/exits $\mathcal B$ in $\partial \mathcal B$ (the circular arc).  If no such arc $\gamma'$ exists, then the entire curve $\gamma$ lifts to $\mathcal B$ and we let $\gamma'=\gamma$.


The arcs labeled $a$ and $b$ in Figure~\ref{NoCylBlock2} project to closed geodesics through the cone point which we also call $a$ and $b$, respectively.  Now observe that $\gamma'$ can be expressed as a concatenation of segments
\[ \gamma' = \gamma_1\gamma_2\cdots\gamma_k,\]
each segment of which either connect consecutive intersection points of $\gamma'$ with $a \cup b$ or connects $\partial \mathcal B$ to an intersection point of $a \cup b$ (with no other intersection points with $a \cup b$ in between).

It is convenient to further lift $\gamma'$ to the holonomy almost trivializing cover $\widetilde{\mathcal B} \to \mathcal B$ illustrated in Figure~\ref{NoCylBlock2:Cover}.  We view $\gamma' = \gamma_1\gamma_2 \cdots \gamma_k$ in either $\mathcal B$ or $\widetilde{\mathcal B}$, whichever is more convenient.  The preimage of $a$ and $b$ is the union of arcs $a_1,a_2,a_3$ and $b_1,b_2,b_3$, respectively, as illustrated.

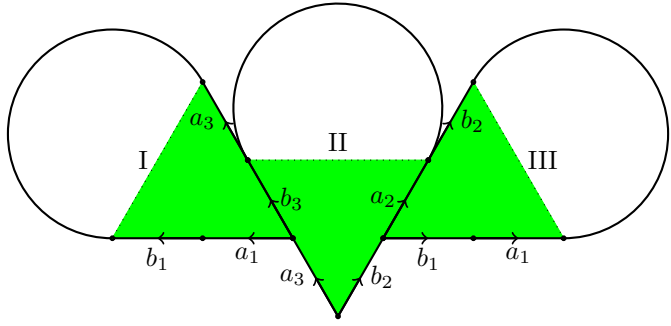
\begin{figure}[ht]
\begin{tikzpicture}[scale=1.2]

\draw[green,fill] (-1.5,-0.866) -- (0.5,-0.866) -- (-0.5,0.866);
\draw[green,fill] (-2,-1.732) -- (-3,0) -- (-1,0);
\draw[green,fill] (-2.5,-0.866) -- (-4.5,-0.866)-- (-3.5,0.866);
\draw[thick] (0.5,-0.866) arc (-90:150:1.157);
\draw[thick] (-4.5,-0.866) arc (270:30:1.157);
\draw[thick] (-1,0) arc (-30:210:1.157);

\draw[thick] (-1.5,-0.866) circle (0.02) -- (-1,0) circle (0.02) -- (-0.5,0.866) circle (0.02);
\draw[thick] (0.5,-0.866) circle (0.02) -- (-0.5,-0.866) circle (0.02) -- (-1.5,-0.866);
\draw[dotted] (0.5,-0.866) -- node[right]{III} (-0.5,0.866);
\draw[thick,->] (-1.5,-0.866) -- (-1.25,-0.433) node[left]{$a_2$};
\draw[thick,->] (-1,0) -- (-0.75,0.433) node[right]{$b_2$};
\draw[thick,->] (-0.5,-0.866) -- (0,-0.866) node[below]{$a_1$};
\draw[thick,->] (-1.5,-0.866) -- (-1,-0.866) node[below]{$b_1$};

\draw[thick] (-2,-1.732) circle (0.02) -- (-1.5,-0.866) circle (0.02) -- (-1,0) circle (0.02);
\draw[thick] (-3,0) circle (0.02) -- (-2.5,-0.866) circle (0.02) -- (-2,-1.732);
\draw[dotted] (-1,0) -- node[above]{II} (-3,0);
\draw[thick,->] (-2,-1.732) -- (-1.75,-1.299) node[right]{$b_2$};
\draw[thick,->] (-2,-1.732) -- (-2.25,-1.299) node[left]{$a_3$};
\draw[thick,->] (-2.5,-0.866) -- (-2.75,-0.433) node[right]{$b_3$};

\draw[thick] (-2.5,-0.866) circle (0.02) -- (-3,0) circle (0.02) -- (-3.5,0.866) circle (0.02);
\draw[thick] (-4.5,-0.866) circle (0.02) -- (-3.5,-0.866) circle (0.02) -- (-2.5,-0.866);
\draw[dotted] (-3.5,0.866) -- node[left]{I} (-4.5,-0.866);
\draw[thick,->] (-2.5,-0.866) -- (-3,-0.866) node[below]{$a_1$};
\draw[thick,->] (-3.5,-0.866) -- (-4,-0.866) node[below]{$b_1$};
\draw[thick,->] (-3,0) -- (-3.25,0.433) node[left]{$a_3$};

\end{tikzpicture}
\caption{The holonomy almost trivializing cover $\widetilde{\mathcal B}$ of $\mathcal B$.}
\label{NoCylBlock2:Cover}
\end{figure}

\begin{claim}  $\gamma'$ must have endpoints on $\partial \mathcal B$, and in the concatenation $\gamma' = \gamma_1\gamma_2 \cdots \gamma_k$, the number of segments $k$ is  at most five.  
\end{claim}
\begin{proof}[Proof of Claim.] Observe that since $\gamma'$ meets the triangle, it must hit at least one of $a$ or $b$.  Consider any intersection point with $a \cup b$.  Appealing to the obvious symmetry interchanging $a$ and $b$, we may assume this intersection point lies in $a$.  Lifting to $\widetilde{\mathcal B}$ (and composing with a covering transformation to choose a different lift if necessary), we can further assume the lift intersects $a_2$. 
By inspecting the figure (and its suggested situation in the plane), we note that if the lift  makes angle between $0$ and $\frac{\pi}{3}$ then it will exit $\widetilde{\mathcal B}$ in region III, while if the angle is between $\frac{\pi}3$ and $\pi$, it will either exit $\widetilde{\mathcal B}$ in region II or pass through $b_3$ and exit $\widetilde{\mathcal B}$ through region I.  In particular, at any point of intersection, we may follow the geodesic in at least one direction (not necessarily both) and in that direction it will exit $\widetilde{\mathcal B}$ through the boundary after intersecting the union of the $a_i$ and $b_j$ in at most one other point.  If $k > 5$, then the point of intersection where $\gamma_3$ meets $\gamma_4$ cannot escape $\widetilde{\mathcal B}$ in two or fewer segments in either direction, and hence $k \leq 5$.
Therefore, there are at most five arcs in the concatenation, completing the proof of the claim.
\end{proof} 

Next, observe that from inside the triangle defining $\mathcal B$ we see four segments labeled $b$, $a$, $b$, $a$ counterclockwise in that order.  For each $1 < i < k$, the arc $\gamma_i$ is contained in the triangle and joins two such labeled segments, but not every pair of segments can be joined.  Specifically since both the first pair and last pair lie along the same side of the triangle, neither of these pairs of segments can be connected by a segment disjoint from the cone point.  Thus, for each $1 < i < k$, $\gamma_i$ either joins the two $a$ segments, the two $b$ segments (which up to symmetry is the same as two $a$ segments), the first and last segments or the middle two (adjacent) segments.  

First suppose that some $\gamma_i$ connects the pair of first and last segments.  In this case, we can show that $\gamma'$ has a self intersection inside the triangle of $\mathcal B$.  To see this, we may lift $\gamma_i$ to $\widetilde{\mathcal B}$ so that it connects $a_3$ and $b_1$: we have labeled this as $\gamma_{i_1}$ in Figure~\ref{NoCylBlock2:Cover2}.  The arc $\gamma_{i_1}$ serves as a ``barrier'', blocking $\gamma'$ from exiting $\mathcal B$ without crossing it.  Indeed, after interchanging the roles of $a$ and $b$, we may assume that another lift to $\widetilde{\mathcal B}$ connects $a_2$ and $b_3$ and extends to intersect the first lift in a point $P_1$ as shown.  This intersection point $P_1$ projects to a self-intersection point of $\gamma'$ inside the triangle, and hence $\gamma$ has a self-intersection point.

Therefore, to prove that $\gamma$ must have a self-intersection point in general, it suffices to consider the case that no $\gamma_i$ connects the first and last segments.  Under these additional assumptions, and upon inspection of the proof of the claim above, we see that in fact $\gamma$ is split into only $k =2$ or $3$ arcs: this is because at any endpoint of the arcs, an arc on one side of this point must terminate at the boundary $\partial \mathcal B$.  

Now there are essentially three cases left to consider: one case where $k = 2$, in which case $\gamma'$ runs from $\partial \mathcal B$ to itself and hits $a$ (without loss of generality) in exactly one point; and two cases where $k =3$, depending on whether there is an arc between two $a$ segments (without loss of generality) or between the adjacent $a$ and $b$ segments.  These cases are illustrated by lifting to $\widetilde{\mathcal B}$ as shown in Figure~\ref{NoCylBlock2:Cover2}, where the arc in each of the cases labeled $\gamma_{i_2}$, $\gamma_{i_3}$, and $\gamma_{i_4}$, respectively.  We briefly describe how to find an intersection point in each of these cases, referring the reader to the figure to clarify the situation.  We note that the last case is the one that determines the circular arc bounding the buffer region.

Suppose that $k =2$, and without loss of generality, $\gamma'$ meets $a \cup b$ in exactly one point of $a$.  This is shown in Figure~\ref{NoCylBlock2:Cover2} with one lifted arc denoted $\gamma_{i_2}$ emanating from a point $x$ on $a_2$.  The other arc of $\gamma'$ can be translated as indicated to emanate from a point $y$ on $a_1$.  Then these arcs intersect in a point denoted $P_2$ in the figure which lies in $I_{x,y}(\frac{2\pi}3)$ (which clearly lies either inside the triangle or the buffer region).

Next, suppose $k =3$ and there is an arc connecting two $a$ edges (without loss of generality).  We lift this to $\widetilde{\mathcal B}$ to an arc denoted $\gamma_{i_3}$ in the figure connecting points on $a_1$ and $a_2$.  This arc really corresponds to {\em two} intersection points of $\gamma'$ with $a$, or one triple point of intersection.  If there's a triple point, then this is a point of self-intersection of $\gamma'$ as required.  Otherwise, one of the points of intersection is closer to the corner of the triangle and one is further.  Similar to the very first case considered (i.e. an arc connecting the pair of first and last segments), the arc $\gamma_{i_3}$ serves as a barrier that one of the other segments of $\gamma'$ must intersect before leaving the triangle.

Finally, suppose $k = 3$ and there is an arc connecting adjacent segments.  In $\widetilde{\mathcal B}$ we can lift this to the arc between $a_3$ and $b_2$ as shown.  The arcs on either side of this arc exit the triangle, and as illustrated on the left of the figure, intersect each other in a point $P_4$ beyond the triangle, at an angle of $\frac{\pi}3$.  If the points of intersection with $a$ and $b$ are $x$ and $y$, respectively, then $P_4$ is contained in the set $I_{x,y}(\frac{\pi}3)$, which is contained inside the buffer region (the figures shows a case where the point is near the boundary of the buffer).  This handles all cases and so completes the proof.  
\end{proof}

\begin{figure}[ht]
\begin{tikzpicture}[scale=1.8]

\draw[thick] (-1.5,-0.866) circle (0.012) -- (-1,0) circle (0.012) -- (-0.5,0.866) circle (0.012);
\draw[thick] (0.5,-0.866) circle (0.012) -- (-0.5,-0.866) circle (0.012) -- (-1.5,-0.866);
\draw[dotted] (0.5,-0.866) -- (-0.5,0.866);
\draw[thick,->] (-1.5,-0.866) -- (-1.25,-0.433) node[left]{$a_2$};
\draw[thick,->] (-1,0) -- (-0.75,0.433) node[left]{$b_2$};
\draw[thick,->] (-0.5,-0.866) -- (0,-0.866) node[below]{$a_1$};
\draw[thick,->] (-1.5,-0.866) -- (-1,-0.866) node[below]{$b_1$};

\draw[thick] (-2,-1.732) circle (0.012) -- (-1.5,-0.866) circle (0.012) -- (-1,0) circle (0.012);
\draw[thick] (-3,0) circle (0.012) -- (-2.5,-0.866) circle (0.012) -- (-2,-1.732);
\draw[dotted] (-1,0) -- (-3,0);
\draw[thick,->] (-2,-1.732) -- (-1.75,-1.299) node[right]{$b_2$};
\draw[thick,->] (-2,-1.732) -- (-2.25,-1.299) node[left]{$a_3$};
\draw[thick,->] (-2.5,-0.866) -- (-2.75,-0.433) node[right]{$b_3$};

\draw[thick] (-2.5,-0.866) circle (0.012) -- (-3,0) circle (0.012) -- (-3.5,0.866) circle (0.012);
\draw[thick] (-4.5,-0.866) circle (0.012) -- (-3.5,-0.866) circle (0.012) -- (-2.5,-0.866);
\draw[dotted] (-3.5,0.866) -- (-4.5,-0.866);
\draw[thick,->] (-2.5,-0.866) -- (-3,-0.866) node[below]{$a_1$};
\draw[thick,->] (-3.5,-0.866) -- (-4,-0.866) node[below]{$b_1$};
\draw[thick,->] (-3,0) -- (-3.25,0.433) node[right]{$a_3$};

\draw[blue, dashed] (0.44,-0.866) arc (0:120:0.96);
\draw[blue] (0.44,-0.866) -- (0.3,-0.2165);
\draw[blue] (0.45,-0.4) node{$P_2$};
\draw[blue] (-1.0433,-0.0433) --  node[above]{$\gamma_{i_2}$} (0.7567,-0.6433);

\draw[brown] (0.25,-0.866) -- node[left]{$\gamma_{i_3}$} (-1.3,-0.5);
\draw[brown] (-0.08,-0.866) node[above]{$P_3$} -- (-0.18,-0.433);

\draw[red] (-4.25,-0.866) -- node[right]{$\gamma_{i_1}$} (-3.125,0.2165);
\draw[red] (-3.45,0.12) node{$P_1$};
\draw[red] (-1.375,-0.6495) -- (-2.875,-0.2165) -- (-3.725,0);

\draw[magenta] (-1.54,-0.9093) -- node[below]{$\gamma_{i_4}$} (-2.46,-0.9093);
\draw[magenta] (-3.46,0.8227) -- 
(-5.45,0.8227);
\draw[magenta] (-4.42,-0.866) -- (-5.45,0.8227) node[left]{$P_4$};

\end{tikzpicture}
\caption{The four cases of geodesics $\gamma_j$ crossing edge $a$ with a self-intersection $P_j$ inside the building block.}
\label{NoCylBlock2:Cover2}
\end{figure}
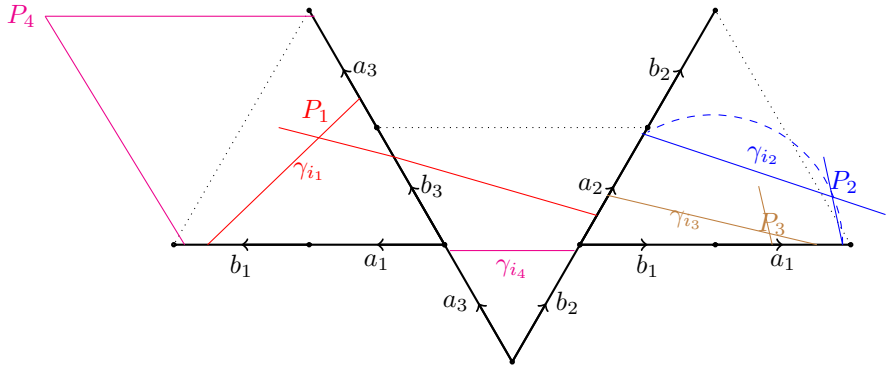

Now consider, for example, the surface on the right in Figure~\ref{NoCyl12gon}.  This is a genus $3$ surface, and in each of the three corners of the big triangle, we see a locally isometrically embedded building block.  The union of the triangle parts of these are all cylinder-free regions.  What's left is an equilateral triangle in the middle, which can clearly contain no cylinders.  Therefore, every cylinder necessarily meets the cylinder-free region and so there are no embedded cylinders.  It is not difficult to construct many more such examples from this idea.


\subsection{Large diameter sets of cylinder curves.}\label{S:qBound}

Here we discuss a construction of fully punctured flat metrics with finite holonomy for which the diameter of $\EC(S)$ is arbitrarily large.  Let $\alpha$ and $\beta$ be simple closed curves on a surface $S$ that fill and so that all intersections have the same sign (i.e.~the geometric and algebraic intersection numbers are the same).  There exists a square-tiled translation structure $\varphi$ on $S$ whose horizontal and vertical foliations are cylinders foliated by curves representing $\alpha$ and $\beta$, respectively.  The number of cone points is at most $2g-2$, by Poincar\'e Hopf.  This flat metric can be constructed from a horizontal cylinder of height $1$ and width $i(\alpha,\beta)$ by gluing the sides and at most $2g-1$ intervals along the top to the intervals along the bottom (these intervals are all horizontal saddle connections, except possibly two adjacent intervals obtained from a saddle connection that must be cut to glue the top to the bottom); see Figure~\ref{SQ}.  

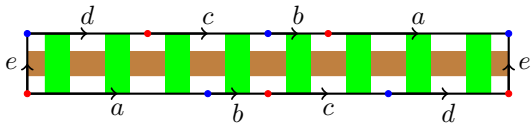
\begin{figure}[ht]
\begin{tikzpicture}[scale=0.8]

\draw[fill,brown] (-4,-0.7) -- (4,-0.7) -- (4,-0.3) -- (-4,-0.3);
\draw[fill,green] (-3.3,0) -- (-3.7,0) -- (-3.7,-1) -- (-3.3,-1);
\draw[fill,green] (-2.3,0) -- (-2.7,0) -- (-2.7,-1) -- (-2.3,-1);
\draw[fill,green] (-1.3,0) -- (-1.7,0) -- (-1.7,-1) -- (-1.3,-1);
\draw[fill,green] (-0.3,0) -- (-0.7,0) -- (-0.7,-1) -- (-0.3,-1);
\draw[fill,green] (0.3,0) -- (0.7,0) -- (0.7,-1) -- (0.3,-1);
\draw[fill,green] (1.3,0) -- (1.7,0) -- (1.7,-1) -- (1.3,-1);
\draw[fill,green] (2.3,0) -- (2.7,0) -- (2.7,-1) -- (2.3,-1);
\draw[fill,green] (3.3,0) -- (3.7,0) -- (3.7,-1) -- (3.3,-1);

\draw[thick] (-4,-1) -- (4,-1) -- (4,0) -- (-4,0) -- (-4,-1);
\draw[thick,->] (-4,-1) -- (-2.5,-1) node[below]{$a$};
\draw[thick,->] (-1,-1) -- (-0.5,-1) node[below]{$b$};
\draw[thick,->] (0,-1) -- (1,-1) node[below]{$c$};
\draw[thick,->] (2,-1) -- (3,-1) node[below]{$d$};
\draw[thick,->] (4,-1) -- (4,-0.5) node[right]{$e$};
\draw[thick,->] (-4,-1) -- (-4,-0.5) node[left]{$e$};
\draw[thick,->] (-4,0) -- (-3,0) node[above]{$d$};
\draw[thick,->] (-2,0) -- (-1,0) node[above]{$c$};
\draw[thick,->] (0,0) -- (0.5,0) node[above]{$b$};
\draw[thick,->] (1,0) -- (2.5,0) node[above]{$a$};

\draw[fill,blue] (4,0) circle (0.05);
\draw[fill,red] (1,0) circle (0.05);
\draw[fill,blue] (0,0) circle (0.05);
\draw[fill,red] (-2,0) circle (0.05);
\draw[fill,blue] (-4,0) circle (0.05);
\draw[fill,red] (4,-1) circle (0.05);
\draw[fill,blue] (2,-1) circle (0.05);
\draw[fill,red] (0,-1) circle (0.05);
\draw[fill,blue] (-1,-1) circle (0.05);
\draw[fill,red] (-4,-1) circle (0.05);

\end{tikzpicture}
\caption{Square-tiled flat metric with a single horizontal cylinder and a single vertical cylinder.}
\label{SQ}
\end{figure}

Note that $\alpha$ and $\beta$ are also curves on $S^\circ$.  Furthermore, for any $d \geq 3$, we can find such a pair of curves so that the distance in the curve graph $\C(S^\circ)$ is at least $d$.  
To see this, pick any $\alpha$ and $\beta$ as above. Then replace $\beta$ with its image under a sufficiently high power of a pseudo-Anosov in the Veech group of $\varphi$ (which exists due to work of Thurston \cite{Thurston} and Veech \cite{Veech}) to complete the argument.  
We therefore assume that the distance in $\C(S^\circ)$ between $\alpha$ and $\beta$ is at least $d$.

Now we deform the flat metric $\varphi^\circ$ to a flat metric $\hat \varphi^\circ$ on $S^\circ$ so that $\alpha$ and $\beta$ are still cylinder curves, but so that the holonomy now has order at least $3$.  To do this, we proceed as follows.  Consider the (at most $2g-1$) horizontal segments along the top and bottom of the rectangle and deform the rectangle into a polygon along these segments to that they are no longer horizontal, but so that each makes a (small) angle with the horizontal that is a rational multiple of $\pi$.  We further assume that sides on the top that are glued to sides on the bottom have the same length.  The vertical sides remain vertical and of the same length.  Any such construction produces a flat metric and if the deformation is small enough, $\alpha$ is still a cylinder curve.  The rationality condition on the angle made with the horizontal ensures that the holonomy is finite; see Figure~\ref{SQdef}.  Taking care with the saddle connection on the top of the rectangle that was cut (if such exists), we can assume that the metric completion has the same number of cone points.

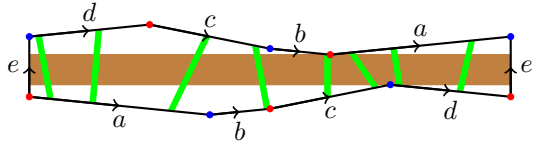
\begin{figure}[ht]
\begin{tikzpicture}[scale=0.8]

\draw[fill,brown] (-4,-0.8) -- (4,-0.8) -- (4,-0.3) -- (-4,-0.3);
\draw[fill,green] (-3.9,0.01) -- (-3.8,0.02) -- (-3.6,-1.04) -- (-3.7,-1.03);
\draw[fill,green] (1.3,-0.27) -- (1.4,-0.26) -- (1.8,-0.84) -- (1.7,-0.86);
\draw[fill,green] (-0.3,-0.14) -- (-0.2,-0.16) -- (0,-1.2) -- (-0.1,-1.21);
\draw[fill,green] (0.9,-0.29) -- (1,-0.3) -- (1,-1) -- (0.9,-1.01);
\draw[fill,green] (-1.1,0.02) -- (-1,0) -- (-1.6,-1.24) -- (-1.7,-1.23);
\draw[fill,green] (3.3,-0.07) -- (3.4,-0.06) -- (3.2,-0.92) -- (3.1,-0.91);
\draw[fill,green] (-2.9,0.11) -- (-2.8,0.12) -- (-2.9,-1.11) -- (-3,-1.1);
\draw[fill,green] (2,-0.2) -- (2.1,-0.19) -- (2.2,-0.82) -- (2.1,-0.81);

\draw[thick] (-4,-1) -- (-1,-1.3) -- (0,-1.2) -- (2,-0.8) -- (4,-1) -- (4,0) -- (1,-0.3) -- (0,-0.2) -- (-2,0.2) -- (-4,0) -- (-4,-1);
\draw[thick,->] (-4,-1) -- (-2.5,-1.15) node[below]{$a$};
\draw[thick,->] (-1,-1.3) -- (-0.5,-1.25) node[below]{$b$};
\draw[thick,->] (0,-1.2) -- (1,-1) node[below]{$c$};
\draw[thick,->] (2,-0.8) -- (3,-0.9) node[below]{$d$};
\draw[thick,->] (4,-1) -- (4,-0.5) node[right]{$e$};
\draw[thick,->] (-4,-1) -- (-4,-0.5) node[left]{$e$};
\draw[thick,->] (-4,0) -- (-3,0.1) node[above]{$d$};
\draw[thick,->] (-2,0.2) -- (-1,0) node[above]{$c$};
\draw[thick,->] (0,-0.2) -- (0.5,-0.25) node[above]{$b$};
\draw[thick,->] (1,-0.3) -- (2.5,-0.15) node[above]{$a$};

\draw[fill,blue] (4,0) circle (0.05);
\draw[fill,red] (1,-0.3) circle (0.05);
\draw[fill,blue] (0,-0.2) circle (0.05);
\draw[fill,red] (-2,0.2) circle (0.05);
\draw[fill,blue] (-4,0) circle (0.05);
\draw[fill,red] (4,-1) circle (0.05);
\draw[fill,blue] (2,-0.8) circle (0.05);
\draw[fill,red] (0,-1.2) circle (0.05);
\draw[fill,blue] (-1,-1.3) circle (0.05);
\draw[fill,red] (-4,-1) circle (0.05);

\end{tikzpicture}
\caption{Deformed flat metric along with the cylinders.}
\label{SQdef}
\end{figure}

To ensure that $\beta$ is still a cylinder curve requires a little more care.  First observe that the holonomy around $\beta$ (oriented ``upward") can be computed as follows: If $\delta_1,\ldots,\delta_k$ are the horizontal saddle connections of $\varphi$, let $n_1,\ldots,n_k$ be the respective number intersection point of $\beta$ with these (so $n_1+\ldots+n_k = i(\alpha,\beta)$).  In the deformed polygon, the saddle connection $\delta_i$ appears as arcs along the top/bottom, each making some angle with the horizontal, and the map gluing the top to the bottom is a composition of a translation and rotation through an angle $\theta_i$ which is the difference of the two angles.  The holonomy of $\beta$ is then a rotation through angle $n_1\theta_1+ \cdots + n_k\theta_k$.  A necessary condition for $\beta$ to be a cylinder curve is that the holonomy is trivial.  If the deformations are small enough, then this is also sufficient.

Now we need to ensure that we can arrange for all of these conditions to be met.  One way to do this is as follows.  Choose a pair of saddle connections, say $\delta_1$ and $\delta_2$.  Along the top, rotate $\delta_1$ by a small positive angle $\psi_1 \in \mathbb Q \pi$ and along the bottom, rotate $\delta_1$ along the bottom by $-\psi_1$, so that the $\delta_1$--arc along the top is glued to that on the bottom by a translation and rotation through angle $\theta_1 = - 2 \psi_1$.  Similarly, along the top, rotate $\delta_2$ by a small {\em negative} angle $\psi_2 \in \mathbb Q \pi$ and along the bottom by $-\psi_2$, so the rotational part of the gluing is through an angle $\theta_2 = -2 \psi_2$.  We assume that $n_1 \theta_1 + n_2 \theta_2 = 0$.  We keep all other saddle connections horizontal (and of the same length) on both the top and bottom.  Observe that the vertical sides can be assumed to remain vertical, but they may have different lengths.  If $\psi_1$ and $\psi_2$ are sufficiently small, then the discrepancy in their lengths is arbitrarily small.  By adjusting the lengths of $\delta_1$, say, we can ensure that the vertical sides have the same length, as required.

Therefore, we can construct a fully punctured flat metric with finite holonomy on $S^\circ$ containing $\alpha$ and $\beta$ which have arbitrarily large distance in $\C(S^\circ)$.




\subsection{Infinite diameter set of cylinder curves} \label{S:slit construction}

Here we describe a construction of $q$--flat metrics with holonomy of order greater than $2$ and with $\EC$ having infinite diameter.  According to Theorem~\ref{T:main}, these examples {\em cannot} be fully punctured.
 
Before we begin, we need one more result due to Klarreich \cite{Kla} (see also \cite{HamBoun,Bom}). To state it, we consider the subspace of measured filling foliations (measured foliations having positive intersection number with every simple closed curve) in the space of all measured foliations, and let $\FL(S)$ denote the quotient space of this subspace by forgetting measures.  We let $\partial \C(S)$ denote the Gromov boundary of the curve graph of $S$.
\begin{theorem} \label{T:Klarreich} There is a mapping class group equivariant homeomorphism
\[ \FL(S) \cong \partial \C(S).\]
Moreover, the bounded length curves along a Teichm\"uller geodesic ray defined by a quadratic differential whose vertical foliation is supported on $\mathcal F \in \FL(S)$ limits to the boundary point corresponding to $\mathcal F$.  \qed
\end{theorem}
In what follows, we use this theorem to make the identification $\partial \C(S)=\FL(S)$.

\begin{theorem} \label{T:6 flat construction}  Given any surface and $\mathcal F \in \FL(S)$, there exist a flat metric $\varphi$ with holonomy of order $6$ for which $\EC(\varphi) \subset \C(S)$ accumulates on $\mathcal F$.  In particular, $\EC(\varphi)$ has infinite diameter.
\end{theorem}
\begin{proof}  By work of Hubbard and Masur \cite{HubMas}, there is a semi-translation structure $\hat \varphi$ on $S$ having vertical foliation supported on $\mathcal F$, up to Whitehead moves and isotopy.  In fact, the Hubbard-Masur result is much stronger, allowing one to specify the underlying conformal structure.  Suppose first that $\hat \varphi$ has a cone point of cone angle $3\pi$ (which is the generic situation). For example, $\hat \varphi$ may be constructed by gluing sides of the polygon as in Figure~\ref{H11}.

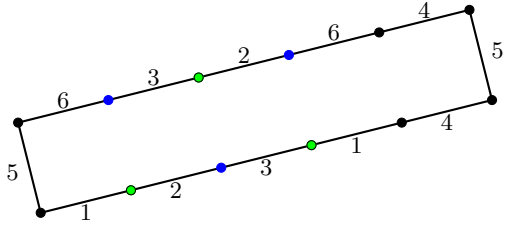
\begin{figure}[ht]
\begin{tikzpicture}[scale=0.3]

\draw[thick] (-1,4) -- (0,0) -- (20,5) -- (19,9) -- (-1,4);
\draw[fill,black] (0,0) circle (0.2);
\draw[fill,green] (4,1) circle (0.2); \draw (4,1) circle (0.2);
\draw[fill,blue] (8,2) circle (0.2);
\draw[fill,green] (12,3) circle (0.2); \draw (12,3) circle (0.2);
\draw[fill,black] (16,4) circle (0.2);
\draw[fill,black] (20,5) circle (0.2);
\draw[fill,black] (19,9) circle (0.2);
\draw[fill,black] (15,8) circle (0.2);
\draw[fill,blue] (11,7) circle (0.2);
\draw[fill,green] (7,6) circle (0.2); \draw (7,6) circle (0.2);
\draw[fill,blue] (3,5) circle (0.2);
\draw[fill,black] (-1,4) circle (0.2);
\node at (2,0) {\small 1};
\node at (6,1) {\small 2};
\node at (10,2) {\small 3};
\node at (14,3) {\small 1};
\node at (18,4) {\small 4};
\node at (-1.25,1.8) {\small 5};
\node at (20.25,7.2) {\small 5};
\node at (1,5) {\small 6};
\node at (5,6) {\small 3};
\node at (9,7) {\small 2};
\node at (13,8) {\small 6};
\node at (17,9) {\small 4};

\end{tikzpicture}
\caption{Semi-translation surface of genus 2 with one cone point of angle $4\pi$ and two of angle $3\pi$.  Sides are glued as indicated, via a translation or translation composed with rotation through angle $\pi$ (to produce an orientable surface).}
\label{H11}
\end{figure}

We use this metric to construct another flat metric with holonomy of order $6$ as follows.  First, at a cone point of cone angle $3\pi$, cut a vertical slit of some small length $\epsilon>0$ along each of the three leaves of $\mathcal F$.  This produces a subsurface $\Sigma \subset S$ which is the complement of triangular disk, equipped with a Euclidean cone metric having three singular points of cone angle $2\pi$ on the boundary.  Cap off the triangular disk by gluing in an equilateral triangle of side length $2 \epsilon$, glued isometrically along the edges, so that the corners of the triangle are glued to the singular points of the boundary.   This produces a flat metric $\varphi$ on $S$, and the one cone point of cone angle $3\pi$ has given rise to three cone points of cone angle $\frac{7\pi}3$ occurring at the corners of the triangle.  The holonomy outside the triangle is exactly the same as before (order $2$), but now loops around these new cone points have holonomy which is a rotation of order $6$.  Therefore, the holonomy of $(S,\varphi)$ has order $6$.

\begin{figure}[ht]
\begin{tikzpicture}[scale=0.3]

\draw[ultra thick,purple] (8,2) -- (8,3.9);
\draw[ultra thick,purple] (3,5) -- (3,3.1);
\draw[ultra thick,purple] (11,7) -- (11,5.1);
\draw[fill,purple] (8,3.9) circle (0.2);
\draw[fill,purple] (3,3.1) circle (0.2);
\draw[fill,purple] (11,5.1) circle (0.2);

\draw[thick] (-1,4) -- (0,0) -- (20,5) -- (19,9) -- (-1,4);
\draw[fill,black] (0,0) circle (0.2);
\draw[fill,green] (4,1) circle (0.2);\draw (4,1) circle (0.2);
\draw[fill,blue] (8,2) circle (0.2);
\draw[fill,green] (12,3) circle (0.2);\draw (12,3) circle (0.2);
\draw[fill,black] (16,4) circle (0.2);
\draw[fill,black] (20,5) circle (0.2);
\draw[fill,black] (19,9) circle (0.2);
\draw[fill,black] (15,8) circle (0.2);
\draw[fill,blue] (11,7) circle (0.2);
\draw[fill,green] (7,6) circle (0.2);\draw (7,6) circle (0.2);
\draw[fill,blue] (3,5) circle (0.2);
\draw[fill,black] (-1,4) circle (0.2);
\node at (2,0) {\small 1};
\node at (6,1) {\small 2};
\node at (10,2) {\small 3};
\node at (14,3) {\small 1};
\node at (18,4) {\small 4};
\node at (-1.25,1.8) {\small 5};
\node at (20.25,7.2) {\small 5};
\node at (1,5) {\small 6};
\node at (5,6) {\small 3};
\node at (9,7) {\small 2};
\node at (13,8) {\small 6};
\node at (17,9) {\small 4};

\draw[ultra thick,purple] (25,2.6) -- (25,6.4) -- (28.3,4.5) -- (25,2.6);
\draw[fill,blue] (25,4.5) circle (0.2);
\draw[fill,blue] (26.65,5.45) circle (0.2);
\draw[fill,blue] (26.65,3.55) circle (0.2);
\draw[fill,purple] (25,2.6) circle (0.2);
\draw[fill,purple] (25,6.4) circle (0.2);
\draw[fill,purple] (28.3,4.5) circle (0.2);

\draw[thick] (33,2.6) -- (33,6.4) -- (36.3,4.5) -- (33,2.6);
\draw[thick] (34.65,3.55) -- node[below]{\small 3} (36.75,2.8);
\draw[thick] (34.65,5.45) -- node[right]{\small 2} (34.4,7.65);
\draw[thick] (33,4.5) -- node[above]{\small 6} (30.6,3.9);

\end{tikzpicture}
\caption{Slits and gluing in a triangle without changing the surface type.  The far right shows the picture near the glued in triangle.}
\label{H11Slit}
\end{figure}
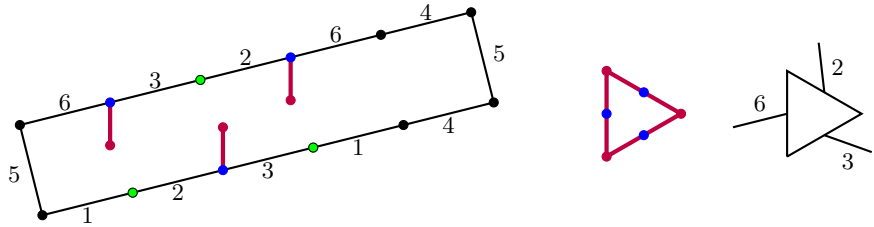

Next, let $g_t \cdot \hat \varphi$ be the flat metric obtained by applying the time = $t$ Teichm\"uller deformation to $\hat \varphi$, so that the underlying conformal structures of $g_t \cdot \hat \varphi$ lie along the Teichm\"uller geodesic defined by $\hat \varphi$; see Section \ref{2:Back}. According to a result of Vorobets \cite[Theorem~1.2]{Vorobets} (generalizing a result of Masur \cite{Mas}), for every $n$ there is a cylinder $A_n$ on $(S,g_n \cdot \hat \varphi)$ whose width (i.e.~distance between boundary components) is greater than some fixed number $\delta >0$.  Since the slit-arc length in $g_n \cdot \hat \varphi$ is $e^{-n} \epsilon$, which is tending to zero, there is a sub-cylinder $A_n' \subset A_n$ which is disjoint from the slit-arcs.  In the metric $\hat \varphi$, the cylinder $A_n'$ is still an embedded cylinder {\em and} it is disjoint from the slit-arcs.  Consequently, these cylinders also embed in $\varphi$.
Because the core curves $a_n$ of $A_n'$ have bounded length on $g_n \cdot \hat \varphi$, it follows that as $n \to \infty$, $\{a_n\} \subset \EC(\varphi)$ converges to $\mathcal F$ by Theorem~\ref{T:Klarreich}.

If $\hat \varphi$ has a cone point with cone angle $k\pi$ with $k \geq 3$, then we can proceed as above, slitting open along vertical arcs of length $\epsilon$ from that cone point.  Instead of gluing in a triangle, we first take a quotient of a triangle with side lengths $2\epsilon$ by the group of order three generated by a rotation.  The result is a Euclidean cone surface, topologically a disk, with a cone point of cone angle $\frac{2\pi}3$ in the interior, and the boundary consisting of a single straight line segment making angle $\frac{\pi}3$ at its endpoints.  Now take a $k$--fold branched cover, branched over the interior cone point, and glue this into the slit open surface as above to produce a new metric $\varphi$.  The arguments above can now be repeated to show that this is a $6$--flat metric (a loop around the ``central" cone point has holonomy order dividing $3$) and cylinder curves $\{a_n\}\subset \EC(\varphi)$ converging to $\mathcal F$.   Finally, if we have a cone point of cone angle $k\pi$, $k \geq 1$ at a completion point where there was a puncture of $S$, we can carry out the exact same construction, except we puncture the branched cover at the interior branch point before we glue it in.
\end{proof}

More generally, given a semi-translation structure $\hat \varphi$ on $S$, we slit open along some union of arcs emanating from cone points and punctures to produce a flat metric on a subsurface $Z \subset S$, so that the complementary subsurfaces are disks or once-punctured disks.  The interiors of the (possibly once-punctured) disks are assumed to be disjoint, however (a) punctures are allowed to lie on the boundary or in the interior and (b) the boundaries of the disks are allowed to intersect or self-intersect in points and/or punctures.    If there are punctures on the boundary, we allow ourselves to ``push" some or all into the interior of the disks (though we do not require that this be done for any of the punctures).   Alternatively, we can view this as perturbing $Z$ slightly (and taking the closure) so that it misses some (possibly empty) subset of the punctures, and we let $Z_0$ be the resulting subsurface.  Observe that after pushing some (or all or none) of the punctures into the disks, the boundaries of these disks are unions of geodesic segments connecting points (and possibly a puncture), and each non-puncture point has total angle $k\pi$, for some $k \geq 2$, in $Z_0$.

\newcommand{\diskreplacement}{
\filldraw[opacity=.2] (0,0) -- (2,-1) -- (2,1) -- (0,0) -- (-1,0) -- (-2,-1) -- (-3,0) -- (-4,-1) -- (-5,-1) -- (-5,1) -- (-4,1) -- (-3,0) -- (-2,1) -- (-1,0) -- (0,0);
\draw (0,0) -- (2,-1) -- (2,1) -- (0,0) -- (-1,0) -- (-2,-1) -- (-3,0) -- (-4,-1) -- (-5,-1) -- (-5,1) -- (-4,1) -- (-3,0) -- (-2,1) -- (-1,0) -- (0,0);
\draw (2,-1) -- (3,-2);
\draw (2,-1) -- (2,-2);
\draw (2,1) -- (3,2);
\draw (-1,0) -- (0,2);
\draw (-2,-1) -- (-2,-2);
\draw (-2,1) -- (-2,2);
\draw (-4,-1) -- (-3,-2);
\draw (-4,-1) -- (-4,-2);
\draw (-5,-1) -- (-6,-2);
\draw (-5,1) -- (-6,2);
\draw (-4,1) -- (-3.5,2);}

\begin{figure}[ht]
\begin{tikzpicture}[scale=.6]

\diskreplacement 
\draw[fill,white] (2,0) circle (0.1);
\draw (2,0) circle (0.1);
\draw[fill,white] (-1,0) circle (0.1);
\draw (-1,0) circle (0.1);
\draw[fill,white] (-4,-1) circle (0.1);
\draw (-4,-1) circle (0.1);

\begin{scope}[shift={(12,0)}]
\diskreplacement 
\draw[fill,white] (1.5,0) circle (0.1);
\draw (1.5,0) circle (0.1);
\draw[fill,white] (-1,0) circle (0.1);
\draw (-1,0) circle (0.1);
\draw[fill,white] (-4,0) circle (0.1);
\draw (-4,0) circle (0.1);
\end{scope}

\end{tikzpicture}
\caption{An example of disks (shaded) in the complement of $Z$ on the left.  Push $Z$ off some of the punctures to produce $Z_0$ on the right.}
\label{F:complementary disks and puncture moves}
\end{figure}
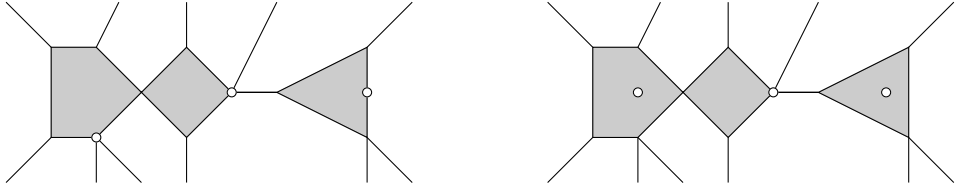

Next, glue in topological disks or punctured disks equipped with Euclidean cone metrics in the complementary components of $Z_0$ having piecewise geodesic boundary, ensuring that any punctures match up in the gluing.   If the resulting Euclidean cone metric $\varphi$ is a $q$--flat metric for some $q$, we say that $\varphi$ is obtained from $\hat \varphi$ by {\em disk replacements}.  The construction of the metric $\varphi$ from $\hat \varphi$ in the proof of Theorem~\ref{T:6 flat construction} is a special case of such disk replacements.  In fact, inspecting the proof we see that in general, if the vertical foliation of $\hat \varphi$ is supported on a filling foliation $\mathcal F \in \FF(S)$ and $\varphi$ is obtained from $\hat \varphi$ by disk replacements, then the proof of Theorem~\ref{T:6 flat construction} also shows that $\EC(\varphi)$ accumulates on $\mathcal F$ in $\partial \C(S)$.

\section{Main Theorem}\label{4:Main}

Having described a variety of examples, we now move on to prove the main theorem.

\bigskip

\noindent {\bf Theorem~\ref{T:main}}
{\em Let $\varphi$ be a fully punctured flat metric on a surface $S$ with finite holonomy of order at least $3$.  Then $\EC(\varphi)$ has finite diameter in $\C(S)$.}\\

We will need a lemma in the proof of Theorem~\ref{T:main}, as well as in the proof of Theorem~\ref{T:the only way} below, but this requires a little set up.

Suppose $\varphi$ is a fully punctured $q$--flat metric on $S$ with $q \geq 3$ and $p \colon (S_0,\varphi_0) \to (S,\varphi)$ the holonomy almost trivializing cover of degree $q_0$ (see Section~\ref{2:Back}).  Since $\varphi$ is fully punctured, so is $\varphi_0$.
Consider the preimage $p^{-1}(A)$ of a maximal embedded cylinder $A \subset S$.   The set $p^{-1}A$ is the union of pairwise disjoint open cylinders (since $A$ is embedded) in equally spaced directions $\theta^1, \cdots ,\theta^{q_0} \in \mathbb R \mathbb P^1$; the covering group is cyclic of order $q_0$ and acts by rotation permuting these directions.  Because $\varphi_0$ is fully punctured, and because the directions of these $q_0$ cylinders are all distinct, it follows that the closures of these cylinders in $S_0$ are pairwise disjoint.   Note, however, that the closure of any cylinder can self-intersect along some saddle connections in the boundary.

\begin{lemma} \label{L:limits of cores}  Let $\varphi$ be a fully punctured flat metric on a surface $S$ with holonomy of order $q \geq 3$ and $p \colon (S_0,\varphi_0) \to (S,\varphi)$ the holonomy almost trivializing cover of degree $q_0$.  Given  a sequence $\{A_n\}$ of distinct, embedded maximal cylinders in $(S,\varphi)$, let $\{A_n^1,\ldots,A_n^{q_0}\}$ be the sequence of components of the preimage in $S_0$.  

Then, there are pairwise disjoint closed subsurfaces $Z^1,\ldots,Z^{q_0}$ of $S_0$, foliated in equally spaced directions $\theta^1,\ldots,\theta^{q_0}$ in $\mathbb R \mathbb P^1$, and a subsequence of $\{A_n\}$ so that for each $j$, (any choice of) a core geodesic $\{a_n^j \subset A_n^j\}$ Hausdorff converges to $Z^j$.  Moreover, for sufficiently large $n$, $a_n^j$ is contained in $Z^j$.
\end{lemma}

\begin{proof}  Because the cylinders are distinct, we may pass to a subsequence so that $\{A_n\}$ all pairwise intersect nontrivially.
 
Let $\theta_n^1, \cdots ,\theta_n^{q_0} \in \mathbb R \mathbb P^1$ be the equally spaced directions of $A_n^1,\ldots,A_n^{q_0}$, respectively.
Pass to a subsequence so that for each $j$, $\{\theta_n^j\}$ is a set of distinct directions and $\theta_n^j \to \theta^j$.  Then $\theta^1,\ldots,\theta^{q_0} \in \mathbb R \mathbb P^1$ are also equally spaced.  Pass to a further subsequence so that the Hausdorff limit $Z^j$ of the closure $\bar{A}_n^j$ exists for each $j$. The covering group $G$ of $p \colon S_0 \to S$ cyclically permutes the set $\{Z^1,\ldots,Z^{q_0}\}$.

Each $Z^j$ is a closed subset of $S_0$ which is a union of geodesics in direction $\theta^j$ that are either closed, or extend maximally in both directions (either as an infinite ray, or a segment limiting to a deleted cone point).  This follows from the fact that each point $x \in Z^j$ is a limit of points in $x_n \in \bar A_n^j$, and the closed core geodesic of $\bar A_n^j$ (or saddle connection in the boundary) through $x_n$ converge to the geodesic in direction $\theta^j$ through $x$.  Since $\bar A_n^1,\ldots,\bar A_n^{q_0}$ are pairwise disjoint, the core geodesics in distinct cylinders from this set have no transverse intersection points.  It follows that geodesics of $Z^1,\ldots,Z^{q_0}$ can have no transverse intersection points.  Because these geodesics are in different directions, this means that $Z^1,\ldots,Z^{q_0}$ are pairwise disjoint.  Since the widths of the cylinders---that is, the distance between their boundary components---tends to zero (by area considerations since the lengths are tending to infinity) $Z^j$ is also a limit of any choice of core geodesic of $A_n^j$. 

Let $\mathcal F_{\theta}$ denote the foliation of $(S_0,q_0)$ in direction $\theta$.  This foliation decomposes into a finite union of maximal cylinders foliated by simple closed geodesics and minimal components; that is, foliated subsurfaces in which each leaf is dense; this well-known fact follows, for example, by passing to measured geodesic laminations via straightening (see \cite{Levitt}) and applying the classification of geodesic laminations into components (see \cite{CB}), though a direct proof can also be obtained from the discussion of the {\em unglue} in \cite[Expos\'e 9]{FLP}.  We refer to both the maximal cylinders and the minimal components as the components of $\mathcal F_\theta$.  From the previous paragraph, it follows that $Z^j$ must be a union of components of $\mathcal F_{\theta^j}$: a priori it is a union of minimal components and subcylinders, but observe that the boundary of the limits must be a union of limits of boundary arcs of $\bar A_n^j$ and therefore must contain cone points (so cylinders in the limit are necessarily maximal).

\begin{claim} For each $j$ and $n$ sufficiently large, $a_n^j \subset Z^j$.
\end{claim}

\begin{proof}
If $Z^j = S_0$ for any $j$, then $q_0 = 1$, and there is nothing to prove, so we assume that $Z^j \neq S_0$.
Fix any $j$ and let $W^j$ denote the closure of the complement $W^j = \overline{S_0 - Z^j}$.  This is a subsurface foliated in direction $\theta^j$ and $Z^j \cap W^j$ is a finite union of saddle connections (in direction $\theta^j$).
Hausdorff convergence implies that for any $\epsilon > 0$, $a_n^j$ is contained in the $\epsilon$--neighborhood $N_\epsilon(Z^j)$ of $Z^j$ for $n$ sufficiently large.  For $\epsilon > 0$ sufficiently small, $U_\epsilon^j = N_\epsilon(Z^j) - Z^j$ is a union of rectangles of width $\epsilon$ together with a (possibly empty) union of sectors of the $\epsilon$--neighborhood of the punctures (note that the union of some of the rectangles making up $U_\epsilon^j$ may form a cylinder, but this does not affect the argument).  These sectors have angles which are integral multiples of $\pi$; See Figure~\ref{F:epsilon-neighborhood}. From this we easily see that if $a_n^j$ meets $U_\epsilon^j$, it must leave $N_\epsilon(Z^j)$, which is impossible.  Therefore, when $a_n^j$ is in $N_\epsilon(Z^j)$, then it is in fact contained in $Z^j$, as required.
\end{proof}

\begin{figure}[ht]
\begin{tikzpicture}[scale=1.1]

\draw[thick] (-2,0) -- (2,0);
\fill [opacity=.2] (-2,0) -- (2,0) -- (2,.2) -- (-2,.2) -- cycle;
\fill [opacity=.5] (0,0) -- (2,0) -- (2.5,-1) -- (1,-1) -- cycle;
\fill [opacity=.6] (0,0) -- (-2,0) -- (-2.5,-1) -- (-1,-1) -- cycle;
\draw[fill] (0,-.7) circle (0.02);
\draw[fill] (-.4,-.6) circle (0.02);
\draw[fill] (.4,-.6) circle (0.02);
\draw[fill] (-.18,-.67) circle (0.02);
\draw[fill] (.18,-.67) circle (0.02);
\draw[thin] (-2,.05) -- (2,.05);
\draw[thin] (-2,.1) -- (2,.1);
\draw[thin] (-2,.15) -- (2,.15);
\draw[thin] (-2,.2) -- (2,.2);
\draw[thin] (-2,.25) -- (2,.25);
\draw[thin] (-2,.3) -- (2,.3);
\draw[thin] (-2,.35) -- (2,.35);
\draw[thin] (-2,.4) -- (2,.4);
\draw[thin] (-2,.45) -- (2,.45);
\draw[thin] (-2,.5) -- (2,.5);
\draw[fill,white] (0,0) circle (0.04);
\draw (0,0) circle (0.04);

\fill [opacity=.2] (5,0) circle (.2);
\fill [white] (5,0) -- (7,0) -- (7,.2) -- (5,.2) -- cycle;
\fill [opacity=.2] (5,0) -- (7,0) -- (7,.2) -- (5,.2) -- cycle;
\draw[thin] (4,-.05) -- (7,-.05);
\draw[thin] (4,-.1) -- (7,-.1);
\draw[thin] (4,-.15) -- (7,-.15);
\draw[thin] (4,-.2) -- (7,-.2);
\draw[thin] (4,-.25) -- (7,-.25);
\draw[thin] (4,-.3) -- (7,-.3);
\draw[thin] (4,.05) -- (7,.05);
\draw[thin] (4,.1) -- (7,.1);
\draw[thin] (4,.15) -- (7,.15);
\draw[thin] (4,.2) -- (7,.2);
\draw[thin] (4,.25) -- (7,.25);
\draw[thin] (4,.3) -- (7,.3);
\draw[thin] (4,.35) -- (7,.35);
\draw[thin] (4,.4) -- (7,.4);
\draw[thin] (4,.45) -- (7,.45);
\draw[thin] (4,.5) -- (7,.5);
\fill [white] (3,-.5) -- (4.5,-.5) -- (5,0) -- (7,0) -- (7.5,-1) -- (2.5,-1) -- cycle;
\fill [opacity=.5] (5,0) -- (7,0) -- (7.5,-1) -- (6,-1) -- cycle;
\draw[thick] (4,0) -- (7,0);
\draw[fill] (4.33,-.37) circle (0.02);
\draw[fill] (4.46,-.49) circle (0.02);
\draw[fill] (5,-.7) circle (0.02);
\draw[fill] (4.6,-.6) circle (0.02);
\draw[fill] (5.4,-.6) circle (0.02);
\draw[fill] (4.82,-.67) circle (0.02);
\draw[fill] (5.18,-.67) circle (0.02);
\draw[fill,white] (5,0) circle (0.04);
\draw (5,0) circle (0.04);

\end{tikzpicture}
\caption{Part of the picture of $N_\epsilon(Z^j)$ near a cone point.  The dark shaded portion is in $Z^j$, while $U_\epsilon^j = N_\epsilon(Z^j) - Z^j$ is the lightly shaded part.  The foliation of $W^j$ is shown horizontally.  The ellipses signify part of the surface that has been omitted.}
\label{F:epsilon-neighborhood}
\end{figure}
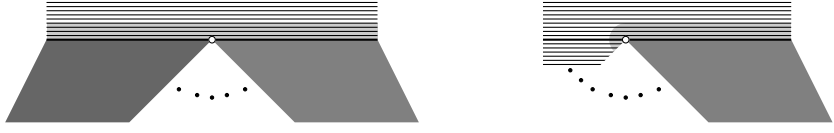

Now observe that since the directions $\{\theta^j_n\}$ are distinct, $a_n^j \subset Z^j$ for sufficiently large $n$, and $\{a_n^j\}$ Hausdorff converges to $Z^j$, it follows that $Z^j_0 = Z^j$.  That is, every saddle connection in direction $\theta^j$ contained in $Z^j$ is already in the subsurface portion $Z^j_0$: otherwise, some saddle connection $\sigma$ would be approximated by an arc of $a_n^j$ for large $n$, and if the saddle connection were not contained in $Z^j_0$, then an arc of $a_n^j$ would be {\em contained in} $\sigma$, contradicting the distinctness of the directions $\{\theta_n^j\}$.  Since $Z^j = Z^j_0$ is a subsurface foliated in direction $\theta^j$, we are done.
\end{proof}

\begin{proof}[Proof of Theorem~\ref{T:main}.] 
Suppose $\{A_n\}$ is a sequence of embedded cylinders in $(S,\varphi)$ so that the sequence of core curves $\{a_n\}$ has infinite diameter in $\C(S)$.  In particular, we may assume that the cylinders are all distinct.
Let $p \colon (S_0,\varphi_0) \to (S,\varphi)$ be the holonomy almost trivializing cover and let $Z^1,\ldots Z^{q_0}$ be the foliated subsurfaces which are Hausdorff limits of core curves $a_n^1,\ldots,a_n^{q_0}$ of $A_n^1,\ldots,A_n^{q_0}$ as in Lemma~\ref{L:limits of cores} (after having passed to a subsequence, if necessary).  Fix an essential simple closed curve $b$ in $Z^2$.

Rafi and Schleimer in \cite{RafSch1} showed that covering maps induce quasi-isometric embeddings between curve complexes. Therefore $\{\widetilde a_n^1\}$ also has infinite diameter in $\C(S_0)$.  
It follows that for $n$ large enough, $a_n^1$ must nontrivially intersect $b$.  But then the Hausdorff limit of the $a_n^1$, which is contained in $Z^1$, necessarily intersects $b$ (though not necessarily transversely).   This is impossible, since $b$ is contained in $Z^2$ and is thus disjoint from $Z^1$.  This contradiction proves that $\{a_n\}$ must have finite diameter, as required.
\end{proof}

Lemma~\ref{L:limits of cores} is also the key tool in proving that when $\EC(\varphi)$ {\em does} have infinite diameter, then it necessarily comes from the kind of construction in Theorem~\ref{T:6 flat construction}, or more generally the one described after its proof.


\begin{theorem} \label{T:the only way} Suppose $\varphi$ is a flat metric on a surface $S$ with holonomy of finite order $q \geq 1$ and $\EC(\varphi)$ has infinite diameter in $\C(S)$.  Then there is a semi-translation structure $\hat \varphi$ so that $\varphi$ is obtained from it by disk replacements.
If in addition $\EC(\varphi)$ accumulates on a point $\mathcal F \in \FF(S) = \partial \C(S)$, then we can assume $\hat \varphi$ has vertical foliation supported on $\mathcal F$.
\end{theorem}




\begin{proof}  Let $\{a_n\} \subset \EC(\varphi)$ be a sequence such that no subsequence has finite diameter in $\C(S)$.  Furthermore, if $\EC(\varphi)$ accumulates on a point $\mathcal F \in \FF(S) = \partial \C(S)$, we assume $a_n \to \mathcal F$, as $n \to \infty$.

Fully puncture $(S,\varphi)$, denote it by $(S^\circ,\varphi^\circ)$, and let $p \colon (S_0,\varphi_0) \to (S^\circ,\varphi^\circ)$ denote the holonomy almost trivializing cover.   Then let $A_n^1,\ldots,A_n^{q_0}$, $Z^1,\ldots,Z^{q_0}$, $\theta^1,\ldots,\theta^{q_0}$ be as in Lemma~\ref{L:limits of cores}, and let us choose $a_n^1,\ldots,a_n^{q_0}$ to be core geodesics for the cylinders so that $a_n^j \to Z^j$ with the foliation in direction $\theta^j$ (having already passed to a subsequence).  Without loss of generality, the directions $\theta_n^j$ of $A_n^j$ are distinct, for all $n$ and $j$ and $a_n^j \subset Z^j$ (by passing to some tail of our subsequence).

Since the $Z^1,\ldots,Z^{q_0}$ are pairwise disjoint and cyclically permuted by the covering group $G$, $p$ restricts to a homeomorphism from each $Z^j$ to  $Z^\circ \subset S^\circ$, a closed subsurface, and maps the foliation by geodesics in direction $\theta^j$ to a foliation $\mathcal F^\circ_0$ of $Z^\circ$ by geodesics: Note that any two of these homeomorphisms differ by an element of the covering group, and thus $(Z^\circ,\mathcal F^\circ_0)$ is independent of $j$.  We also note that $Z^\circ$ is the Hausdorff limit of the closures $\{\bar A_n \cap S^\circ\}$ and that the core geodesics of $A_n$ limit to the geodesics in the foliation $\mathcal F^\circ_0$.  Let $Z = \bar Z^\circ \subset S$, and observe that this is an embedded subsurface, except possibly at the cone points, also with a (singular) foliation $\mathcal F_0$ by limits of the core curve $a_n$ of $A_n$.   Moreover, since $a_n^j \subset Z^j$, it follows that $a_n \subset Z$.  Since $\{a_n\}$ is assumed to have infinite diameter in $\C(S)$, the boundary of $Z$ must be homotopically inessential in $S$---otherwise the boundary would be an essential curve disjoint from all $a_n$, proving $\{a_n\}_{n > N}$ has diameter $2$ in $\C(S)$.  But then, the complement of $Z$ is a collection of disks, each containing at most one puncture (possibly on its boundary).  

For each complementary (possibly once-punctured) disk, we now describe how to glue the boundary arcs together, thus ``collapsing" $Z$ onto $S$, and producing a semi-translation structure $\hat \varphi$ for which $\varphi$ is therefore obtained by disk replacements.  First, fix a component without a puncture in its interior.  Observe that the boundary is a union of finitely many geodesic segments connecting points for which the total angle inside $Z$ has the form $k \pi$ for some integer $k \geq 2$, together with at most one puncture; call these the {\em corners} of the disk.

Let the geodesic arcs be denoted $\alpha_1,\ldots,\alpha_n$ and let $\ell_j$ be the length of $\alpha_j$.  Since each arc is a geodesic, we must have
\begin{equation} \label{E:gluing eqn}
\ell_j \leq \sum_{i \neq j} \ell_i 
\end{equation}
for all $j$.  To see this, note that otherwise some $\ell_j$ is greater than the sum of the remaining lengths, in which case the concatenation of these remaining arcs is a path of shorter length than a geodesic which is homotopic to the geodesic via the disk, a contradiction (if there's a puncture, it's necessarily on the boundary by assumption, and we can do a small perturbation to avoid it and produce the same contradiction).  Therefore the inequality \eqref{E:gluing eqn} holds.

After reindexing, suppose $\ell_1 \geq \ell_i$ and $\ell_2 \leq \ell_i$, for all $i$.  For $0 < t \leq \ell_2/2$, we can isometrically ``fold together" arcs of length $t$ emanating from any corner.  Doing this at each corner we have reduced the lengths of each arc by $2t$, so that the lengths are now $\ell_j - 2t$, for each $j$.  If there is some $0 < t_0 \leq \ell_2/2$ for which
\[ \ell_1 - 2t_0 = \sum_{i\neq 1} (\ell_i - 2t_0),\]
then perform the fold for this choice of $t_0$, and we can isometrically glue the remaining portion of $\alpha_1$ of length $\ell_1-2t_0$ to the union of the others, and we have produced the required collapse.  If there is no such $t_0$, then we perform the fold with $t = \ell_2/2$, so that the entire side $\alpha_2$ has been glued (others may have been glued as well).  We now have a new disk {\em with fewer sides}, and since $\alpha_1$ was the longest, our assumption about the nonexistence of $t_0$ implies that the new lengths also satisfy \eqref{E:gluing eqn} for all indices $j$ (in our new index set).  Continuing inductively, all arcs of the disk are glued as required to produce the metric $\hat \varphi$ on $S$.

If there is a puncture in the interior of the disk, and inequality \eqref{E:gluing eqn} holds, we carry out the procedure just described, placing the puncture at the final point of the gluing.  If \eqref{E:gluing eqn} does not hold, then assuming $\alpha_1$ is longest side again, push the puncture onto the midpoint of that side.  This subdivides $\alpha_1$ into $\alpha_1'$ and $\alpha_1''$, each of half the length, $\ell_1' = \ell_1'' = \ell_1/2$.  Now observe that \eqref{E:gluing eqn} holds for the new lengths $\ell_1',\ell_1'',\ell_2,\ldots,\ell_n$, and we can proceed as above.

We have therefore produced a metric $\hat \varphi$ foliated by $\mathcal F_0$, hence defining a semi-translation structure, for which $\varphi$ is obtained by disk replacements.  Finally, if we were in the situation that $\{a_n\}$ accumulated on $\mathcal F \in \FF(S) = \partial \C(S)$, then since $\mathcal F$ is filling, we must have $\mathcal F_0 = \mathcal F$ by Theorem~\ref{T:Klarreich}.  The proof of this last equality requires a bit more care, and one way to see it is as follows.  Straightening all curves and foliations to geodesics and geodesic laminations (with respect to a fixed hyperbolic structure), we see that the geodesic representatives $a_n^*$ of $a_n$ must have a subsequence that Hausdorff converges to a lamination having no transverse intersections with the straightening of $\mathcal F_0$.  On the other hand, any Hausdorff accumulation point of $a_n^*$ must also contain the straightening of $\mathcal F$ (see \cite{HamBoun,Bom}), and the straightenings of $\mathcal F$ and $\mathcal F_0$ have no transverse intersections.  Since $\mathcal F$ is filling, it follows that $\mathcal F_0 = \mathcal F$.
\end{proof}

Finally, we note that following similar ideas to the proof of Theorem~\ref{T:main}, we can prove the following about the set of embedded saddle connections, $\ESC(\varphi)$ for a flat metric with holonomy of order at least $3$.

\begin{theorem} \label{T:ESC} Let $\varphi$ be a fully punctured flat metric on a surface $S$ with finite holonomy of order $q \geq 3$.  Then $\ESC(\varphi)$ has finite diameter in $\A(S)$.
\end{theorem}

\begin{proof} Suppose there exists a sequence $\{a_n\} \subset \ESC(\varphi)$ so that the distance from $a_n$ to some point of $\A(S)$ tends to infinity.  Pass to the holonomy almost trivializing cover $p \colon (S_0,\varphi_0) \to (S,\varphi)$, let $a_n^1 \cup \ldots \cup a_n^{q_0} = p^{-1}(a_n)$ be the preimage set of pairwise disjoint embedded saddle connections.  If $\theta_n^j$ is the direction of $a_n^j$, then after passing to a subsequence we may assume $\theta_n^j \to \theta^j$ and that $a_n^j$ Hausdorff converges to a closed subset $Z^j$, for each $j = 1,\ldots,q_0$.  The proof of the lemma can be carried out verbatim to show that $Z^1,\ldots,Z^{q_0}$ consists of pairwise disjoint embedded subsurfaces foliated in equally spaced directions $\theta^1,\ldots,\theta^{q_0}$. 

As in the proof of Theorem~\ref{T:the only way}, $p$ restricts to a homeomorphism $Z^j \to Z \subset S$, for any component $Z^j$.  Because $Z^1,\ldots,Z^{q_0}$ is not the entire surface $S_0$ (they are disjoint and foliated in different directions) it follows that $Z \neq S$.  In particular, there is a saddle connection in the boundary of $Z$.  Again as in the proof of Theorem~\ref{T:the only way}, we find that $a_n$ Hausdorff converges to $Z$ and is eventually contained in $Z$ for $n$ large enough.  Consequently, the saddle connection in the boundary of $Z$ is disjoint from $a_n$, for large enough $n$, contradicting the fact that the distance to some point of $\A(S)$ tends to infinity.
\end{proof}



\section{Questions}

The theorems and examples of this paper suggest a number of questions.  We list a couple of them here.

An affirmative answer to the following question would provide a quantitative version of Theorem~\ref{T:main}.
\begin{question} Given a fully punctured flat metric $\varphi$ with holonomy of finite order $q \geq 3$, is there a bound on the diameter of $\EC(\varphi)$ in $\C(S)$ that depends only $q$ and $S$?  Is there a bound on the diameter of $\ESC(\varphi)$ in $\A(S)$ that depends only on $q$ and $S$?
\end{question}
The examples in Section~\ref{S:qBound} show that if there is a bound on the diameter of $\EC(\varphi)$, it must necessarily depend on $q$.\\

The cylinder curves in the example from Section~\ref{S:slit construction} limit to a single point on the Gromov boundary of the curve complex.
\begin{question} Do there exist flat metrics $\varphi$ with holonomy of finite order $q \geq 3$ so that $\EC(\varphi) \subset \C(S)$ accumulates on more than one point in the Gromov boundary?
\end{question}
The proof of Theorem~\ref{T:the only way} picks out a subsurface (with disk complementary components) and foliation when $\EC(\varphi)$ has infinite diameter, suggesting that the answer to this question is negative.  

\begin{question} \label{Q:is cone point blow up enough} If $\varphi$ is obtained from $\hat \varphi$ by a disk replacement, does $\EC(\varphi)$ have infinite diameter?  If not, does it at least contain infinitely many distinct curves?
\end{question}



\bibliography{qdiff}

\begin{bibdiv}
\begin{biblist}

\bib{Ban}{article}{
      author={Bankovic, Anja},
       title={Horowitz-{R}andol pairs of curves in {$q$}-differential metrics},
        date={2014},
        ISSN={1472-2747},
     journal={Algebr. Geom. Topol.},
      volume={14},
      number={5},
       pages={3107\ndash 3139},
         url={https://doi.org/10.2140/agt.2014.14.3107},
      review={\MR{3276858}},
}

\bib{BGKT}{article}{
      author={Boshernitzan, M.},
      author={Galperin, G.},
      author={Kr\"{u}ger, T.},
      author={Troubetzkoy, S.},
       title={Periodic billiard orbits are dense in rational polygons},
        date={1998},
        ISSN={0002-9947},
     journal={Trans. Amer. Math. Soc.},
      volume={350},
      number={9},
       pages={3523\ndash 3535},
         url={https://doi.org/10.1090/S0002-9947-98-02089-3},
      review={\MR{1458298}},
}

\bib{Bow}{article}{
      author={Bowditch, Brian~H.},
       title={Intersection numbers and the hyperbolicity of the curve complex},
        date={2006},
        ISSN={0075-4102},
     journal={J. Reine Angew. Math.},
      volume={598},
       pages={105\ndash 129},
         url={https://doi.org/10.1515/CRELLE.2006.070},
      review={\MR{2270568}},
}

\bib{CB}{book}{
      author={Casson, Andrew~J.},
      author={Bleiler, Steven~A.},
       title={Automorphisms of surfaces after {N}ielsen and {T}hurston},
      series={London Mathematical Society Student Texts},
   publisher={Cambridge University Press, Cambridge},
        date={1988},
      volume={9},
        ISBN={0-521-34203-1},
         url={https://doi.org/10.1017/CBO9780511623912},
      review={\MR{964685}},
}

\bib{DRT}{article}{
      author={Disarlo, Valentina},
      author={Randecker, Anja},
      author={Tang, Robert},
       title={Rigidity of the saddle connection complex},
        date={2018},
     journal={Preprint, {\texttt{arXiv:1810.00961}}},
}

\bib{DLR}{article}{
      author={Duchin, Moon},
      author={Leininger, Christopher~J.},
      author={Rafi, Kasra},
       title={Length spectra and degeneration of flat metrics},
        date={2010},
        ISSN={0020-9910},
     journal={Invent. Math.},
      volume={182},
      number={2},
       pages={231\ndash 277},
         url={https://doi.org/10.1007/s00222-010-0262-y},
      review={\MR{2729268}},
}

\bib{FLP}{book}{
      author={Fathi, Albert},
      author={Laudenbach, Fran\c{c}ois},
      author={Po\'{e}naru, Valentin},
       title={Thurston's work on surfaces},
      series={Mathematical Notes},
   publisher={Princeton University Press, Princeton, NJ},
        date={2012},
      volume={48},
        ISBN={978-0-691-14735-2},
        note={Translated from the 1979 French original by Djun M. Kim and Dan
  Margalit},
      review={\MR{3053012}},
}

\bib{Gardiner}{book}{
      author={Gardiner, Frederick~P.},
       title={Teichm\"{u}ller theory and quadratic differentials},
      series={Pure and Applied Mathematics (New York)},
   publisher={John Wiley \& Sons, Inc., New York},
        date={1987},
        ISBN={0-471-84539-6},
        note={A Wiley-Interscience Publication},
      review={\MR{903027}},
}

\bib{HamBoun}{incollection}{
      author={Hamenst\"{a}dt, Ursula},
       title={Train tracks and the {G}romov boundary of the complex of curves},
        date={2006},
   booktitle={Spaces of {K}leinian groups},
      series={London Math. Soc. Lecture Note Ser.},
      volume={329},
   publisher={Cambridge Univ. Press, Cambridge},
       pages={187\ndash 207},
      review={\MR{2258749}},
}

\bib{HubMas}{article}{
      author={Hubbard, John~H.},
      author={Masur, Howard},
       title={On the existence and uniqueness of {S}trebel differentials},
        date={1976},
        ISSN={0002-9904},
     journal={Bull. Amer. Math. Soc.},
      volume={82},
      number={1},
       pages={77\ndash 79},
  url={https://doi-org.ezproxy.rice.edu/10.1090/S0002-9904-1976-13965-1},
      review={\MR{393583}},
}

\bib{Kla}{article}{
      author={Klarreich, Erika},
       title={The boundary at infinity of the curve complex and the relative
  {T}eichm{\"u}ller space},
     journal={preprint, {\texttt{arXiv:1803.10339}}},
}

\bib{Levitt}{article}{
      author={Levitt, Gilbert},
       title={Foliations and laminations on hyperbolic surfaces},
        date={1983},
        ISSN={0040-9383},
     journal={Topology},
      volume={22},
      number={2},
       pages={119\ndash 135},
         url={https://doi.org/10.1016/0040-9383(83)90023-X},
      review={\MR{683752}},
}

\bib{Loving}{article}{
      author={Loving, Marissa},
       title={Length spectra of flat metrics coming from {$q$-}differentials},
        date={2018},
     journal={To appear in {G}roups, {G}eometry, and {D}ynamics.
  {\texttt{arxiv/1810.01793}}},
}

\bib{Mas}{article}{
      author={Masur, Howard},
       title={Closed trajectories for quadratic differentials with an
  application to billiards},
        date={1986},
        ISSN={0012-7094},
     journal={Duke Math. J.},
      volume={53},
      number={2},
       pages={307\ndash 314},
         url={https://doi.org/10.1215/S0012-7094-86-05319-6},
      review={\MR{850537}},
}

\bib{MM1}{article}{
      author={Masur, Howard~A.},
      author={Minsky, Yair~N.},
       title={Geometry of the complex of curves. {I}. {H}yperbolicity},
        date={1999},
        ISSN={0020-9910},
     journal={Invent. Math.},
      volume={138},
      number={1},
       pages={103\ndash 149},
         url={https://doi.org/10.1007/s002220050343},
      review={\MR{1714338}},
}

\bib{MinTay}{article}{
      author={Minsky, Yair~N.},
      author={Taylor, Samuel~J.},
       title={Fibered faces, veering triangulations, and the arc complex},
        date={2017},
        ISSN={1016-443X},
     journal={Geom. Funct. Anal.},
      volume={27},
      number={6},
       pages={1450\ndash 1496},
         url={https://doi.org/10.1007/s00039-017-0430-y},
      review={\MR{3737367}},
}

\bib{Nguyen}{article}{
      author={Nguyen, Duc-Manh},
       title={Translation surfaces and the curve graph in genus two},
        date={2017},
        ISSN={1472-2747},
     journal={Algebr. Geom. Topol.},
      volume={17},
      number={4},
       pages={2177\ndash 2237},
         url={https://doi.org/10.2140/agt.2017.17.2177},
      review={\MR{3685606}},
}

\bib{Pan}{article}{
      author={Pan, Huiping},
       title={Affine equivalence and saddle connection graphs of translation
  surfaces},
        date={2018},
     journal={Preprint, {\texttt{arXiv:1809.06248}}},
}

\bib{Bom}{article}{
      author={Pho-On, Witsarut},
       title={Infinite unicorn paths and {G}romov boundaries},
        date={2017},
        ISSN={1661-7207},
     journal={Groups Geom. Dyn.},
      volume={11},
      number={1},
       pages={353\ndash 370},
         url={https://doi.org/10.4171/GGD/399},
      review={\MR{3641844}},
}

\bib{RafSch1}{article}{
      author={Rafi, Kasra},
      author={Schleimer, Saul},
       title={Covers and the curve complex},
        date={2009},
        ISSN={1465-3060},
     journal={Geom. Topol.},
      volume={13},
      number={4},
       pages={2141\ndash 2162},
         url={https://doi.org/10.2140/gt.2009.13.2141},
      review={\MR{2507116}},
}

\bib{Str}{book}{
      author={Strebel, Kurt},
       title={Quadratic differentials},
      series={Ergebnisse der Mathematik und ihrer Grenzgebiete (3) [Results in
  Mathematics and Related Areas (3)]},
   publisher={Springer-Verlag, Berlin},
        date={1984},
      volume={5},
        ISBN={3-540-13035-7},
         url={https://doi.org/10.1007/978-3-662-02414-0},
      review={\MR{743423}},
}

\bib{TangPrep}{article}{
      author={Tang, Robert},
       title={The geometry of cylinder graphs for half-translation surfaces},
     journal={in preparation},
}

\bib{TW}{article}{
      author={Tang, Robert},
      author={Webb, Richard C.~H.},
       title={Shadows of {T}eichm\"{u}ller discs in the curve graph},
        date={2018},
        ISSN={1073-7928},
     journal={Int. Math. Res. Not. IMRN},
      number={11},
       pages={3301\ndash 3341},
         url={https://doi.org/10.1093/imrn/rnw318},
      review={\MR{3810220}},
}

\bib{Thurston}{article}{
      author={Thurston, William~P.},
       title={On the geometry and dynamics of diffeomorphisms of surfaces},
        date={1988},
        ISSN={0273-0979},
     journal={Bull. Amer. Math. Soc. (N.S.)},
      volume={19},
      number={2},
       pages={417\ndash 431},
         url={https://doi.org/10.1090/S0273-0979-1988-15685-6},
      review={\MR{956596}},
}

\bib{Veech}{article}{
      author={Veech, W.~A.},
       title={Teichm\"{u}ller curves in moduli space, {E}isenstein series and
  an application to triangular billiards},
        date={1989},
        ISSN={0020-9910},
     journal={Invent. Math.},
      volume={97},
      number={3},
       pages={553\ndash 583},
         url={https://doi.org/10.1007/BF01388890},
      review={\MR{1005006}},
}

\bib{Vorobets}{incollection}{
      author={Vorobets, Yaroslav},
       title={Periodic geodesics on generic translation surfaces},
        date={2005},
   booktitle={Algebraic and topological dynamics},
      series={Contemp. Math.},
      volume={385},
   publisher={Amer. Math. Soc., Providence, RI},
       pages={205\ndash 258},
         url={https://doi.org/10.1090/conm/385/07199},
      review={\MR{2180238}},
}

\end{biblist}
\end{bibdiv}

\end{document}